\documentclass[a4paper,12pt,onecolumn,twoside]{scrartcl}

\usepackage[english]{babel}
\usepackage[utf8]{inputenc}
\usepackage{uniinput}
\DeclareUnicodeCharacter{03B9}{{{\ensuremath{\iota}\hspace{0.3mm}}}}

\usepackage{amssymb}
\usepackage{tikz}
\usetikzlibrary{shapes,fit,matrix,calc,positioning, decorations.pathreplacing,decorations.pathmorphing} 

\usepackage{amsmath} 

\usepackage{mathtools}

\usepackage[colorlinks,linkcolor=black,citecolor=black,urlcolor=black,hyperfootnotes=true,breaklinks]{hyperref} 
\usepackage{cleveref}
\usepackage{amsthm}

\usepackage{microtype}

\usepackage{mathrsfs}
\usepackage{leftidx}

\usepackage{enumerate}

\newcommand{\longmapsfrom}{\reflectbox{$\longmapsto$}}

\newcommand{\Z}{\mathbb{Z}}

\renewcommand{\div}{\operatorname{div}}
\DeclareMathOperator{\Ext}{Ext}
\DeclareMathOperator{\End}{End}
\DeclareMathOperator{\Hom}{Hom}
\newcommand{\Hm}{{\operatorname{H}}} 

\DeclareMathOperator{\im}{im}
\newcommand{\Mod}{{\tt Mod}}
\DeclareMathOperator{\pres}{PRes}

\newcommand{\kommentar}[1]{}

\newcommand{\ls}{\lfloor}
\newcommand{\rs}{\rfloor}
\newcommand{\dum}{\sum\nolimits} 

\newcommand{\lpt}{\langle\hspace{-3pt}\langle} 
\newcommand{\rpt}{\rangle\hspace{-3pt}\rangle}

\newcommand{\natop}[2]{\genfrac{}{}{0pt}{}{#1}{#2}}

\newcommand{\smpl} 
{
\linethickness{0.05mm}
\begin{picture}(9,0)
\bezier{10}(5,1)(5,5)(5,9)
\bezier{10}(1,5)(5,5)(9,5)
\end{picture}
\thinlines
}
\newcommand{\smmi} 
{
\linethickness{0.05mm}
\begin{picture}(9,0)
\bezier{10}(1,5)(5,5)(9,5)
\end{picture}
\thinlines
}

\allowdisplaybreaks	
\newcommand{\nsection}[1]{\vfill\eject
\section{#1}}

\newcommand{\ru}[1]{\rule[#1mm]{0mm}{0mm}} 

\newcommand{\asm}[1]{{$\textstyle \vphantom{\substack{1\\1\\1}} #1$}}
\newcommand{\nsm}[1]{\asm{#1}}

\newcommand{\nasm}[1]{{$\textstyle #1$}}

\newcommand{\ovs}[2][=]{\overset{\hphantom{\displaystyle{#1}}\mathllap{#2}}{#1}\,}
\newcommand{\eqs}{=\,}
\newcommand{\eqsp}{\,}
\newcommand{\svs}[2][=]{\overset{\mathclap{#2}}{#1}}
\newcommand{\sqs}{=}

\newcommand{\funcblc}{-}

\theoremstyle{definition}
\newtheorem{Def}{Definition} \crefname{Def}{Definition}{Definitions}
\newtheorem{bem}[Def]{Remark} \crefname{bem}{Remark}{Remarks}
\newtheorem{bsp}[Def]{Example} \crefname{bsp}{Example}{Examples}
 \crefname{defrem}{Definition/Remark}{FIXME}

\theoremstyle{plain}
\newtheorem{lemma}[Def]{Lemma} \crefname{lemma}{Lemma}{Lemmas}
\newtheorem{kor}[Def]{Corollary} \crefname{kor}{Corollary}{Corollaries}
\newtheorem{pp}[Def]{Proposition} \crefname{pp}{Proposition}{Propositions}
\newtheorem{tm}[Def]{Theorem} \crefname{tm}{Theorem}{Theorems}

\newcommand{\zp}{\mathbb{Z}_{(p)}}
\newcommand{\fp}{ {\mathbb{F}_{\!p}} }
\newcommand{\fs}{ {\mathbb{F}_{\!2}} }
\newcommand{\ff}{ {\mathbb{F}_{\!5}} }
\DeclareMathOperator{\Sy}{S}
\newcommand{\Sp}{\Sy_p}
\newcommand{\A}{{\rm A}}
\DeclareMathOperator{\pr}{Pr}

\newcommand{\name}[1]{\textsc{#1}}

\begin{document}
\newcommand{\xyprot}{}

\title{An $\A_∞$-structure on the cohomology ring of the symmetric group $\Sp$ with coefficients in $\fp$}
\author{Stephan Schmid}
\maketitle

\begin{small}
\begin{quote}
\begin{center}{\bf Abstract}\end{center}\vspace*{2mm}
Let $p$ be a prime. Let $\fp\!\Sp$ be the group algebra of the symmetric group over the finite field $\fp$ with $|\fp|=p$. Let $\fp$ be the trivial $\fp\!\Sp$-module. We present a projective resolution $\pres \fp$ of the module $\fp$ and equip the Yoneda algebra $\Ext^*_{\fp\!\Sp}(\fp,\fp)$ with an $\A_∞$-structure such that $\Ext^*_{\fp\!\Sp}(\fp,\fp)$ becomes a minimal model of the dg-algebra $\Hom^*_{\fp\!\Sp}(\pres \fp, \pres \fp)$.
\end{quote}
\end{small}


\renewcommand{\thefootnote}{\fnsymbol{footnote}}
\renewcommand{\thefootnote}{\arabic{footnote}}

\tableofcontents

\subsection{Introduction}
\paragraph{$\A_∞$-algebras}
Let $R$ be a commutative ring.  Let $A$ be a $\Z$-graded $R$-module. Let $m_1:A→A$ be a graded map of degree $1$ with $m_1^2=0$, i.e.\ a differential on $A$. Let $m_2:A\otimes A→A$ be a graded map of degree $0$ satisfying the Leibniz rule, i.e.\ 
\begin{align*}
m_1\circ m_2 = m_2\circ(m_1\otimes 1 + 1\otimes m_1).
\end{align*}
The map $m_2$ is in general not required to be associative. Instead, we require that for a morphism $m_3:A^{\otimes 3}→A$, the following identity holds.
\[m_2\circ (m_2\otimes 1 - 1\otimes m_2) = m_1\circ m_3 + m_3\circ(m_1\otimes 1^{\otimes 2} + 1\otimes m_1 \otimes 1 + 1^{\otimes 2} \otimes  m_1)\]
Following \name{Stasheff}, cf.\ \cite{St63}, this can be continued in a certain way with higher multiplication maps to obtain a tuple of graded maps $(m_n:A^{\otimes n}→ A)_{n\geq 1}$ of certain degrees satisfying the Stasheff identities, cf.\ e.g.\ \eqref{ainfrel}. The tuple $(A, (m_n)_{n\geq 1})$ is then called an $\A_∞$-algebra. 

A morphism of $\A_∞$-algebras from $(A',(m'_n)_{n\geq 1})$ to $(A,(m_n)_{n\geq 1})$ is a tuple of graded maps \mbox{$(f_n:A'^{\otimes n}→ A)_{n\geq 1}$}  of certain degrees satisfying the identities \eqref{finfrel}. The first two of these are
\begin{align*}
\eqref{finfrel}[1]: && f_1 \circ m'_1 \eqs & m_1 \circ f_1\\
\eqref{finfrel}[2]: && f_1\circ m'_2 - f_2\circ(m'_1\otimes 1 + 1\otimes m'_1) \eqs& m_1\circ f_2 + m_2\circ (f_1\otimes f_1).
\end{align*}


So a morphism $f=(f_n)_{n\geq 1}$ of $\A_∞$-algebras from $(A',(m'_n)_{n \geq 1})$ to $(A,(m_n)_{n \geq 1})$ contains a morphism of complexes $f_1:(A',m_1')→(A,m_1)$. We say that $f$ is a quasi-isomorphism of $\A_∞$-algebras if $f_1$ is a quasi-isomorphism. Furthermore, there is a concept of homotopy for  $\A_∞$-morphisms, cf.\ e.g.\ \cite[3.7]{Ke01} and \cite[Définition 1.2.1.7]{Le03}. 

\paragraph{History}
The history of $\A_∞$-algebras is outlined in \cite{Ke01} and \cite{Ke01ad}.

As already mentioned, \name{Stasheff} introduced $\A_∞$-algebras in 1963.

If $R$ is a field, $\mathbb{F}:=R$, we have the following basic results on $\A_∞$-algebras, which are known since the early 1980s.
\begin{itemize}
\item Each quasi-isomorphism of $\A_∞$-algebras is a homotopy equivalence, cf.\ \cite{Pr84}, \cite{Ka87}, …  
\item The minimality theorem: Each $\A_∞$-algebra $(A,(m_n)_{n\geq 1})$ is quasi-isomorphic to an $\A_∞$-algebra  $(A',\{m'_n\}_{n \geq 1})$ with $m'_1=0$, cf.\ \cite{Ka82}, \cite{Ka80}, \cite{Pr84}, \cite{GuLaSt91}, \cite{JoLa01}, \cite{Me99}, … . The $\A_∞$-algebra  $A'$  is then called a minimal model of $A$.
\end{itemize}

%



Suppose given an $\mathbb{F}$-algebra $B$ and suppose given an $B$-module $M$ together with a projective resolution $\pres M$ of $M$.
The homology of the dg-algebra $\Hom^*_B(\pres M,\pres M)$ is the Yoneda algebra $\Ext^*_B(M,M)$. By the minimality theorem, it is possible to construct an $\A_∞$-structure on $\Ext^*_B(M,M)$ such that $\Ext^*_B(M,M)$ becomes a minimal model of the dg-algebra $\Hom^*_B(\pres M,\pres M)$. 
For the purpose of this introduction, we will call such an $\A_∞$-structure on $\Ext^*_B(M,M)$ the canonical $\A_∞$-structure on $\Ext^*_B(M,M)$, which is unique up to isomorphisms of $\A_∞$-algebras, cf.\ \cite[3.3]{Ke01}.


This structure has been calculated or partially calculated in several cases.

Let $p$ be a prime.

For an arbitrary field $\mathbb{F}$, \name{Madsen} computed the canonical $\A_∞$-structure on  $\Ext^*_{\mathbb{F}[α]/(α^n)}(\mathbb{F},\mathbb{F})$, where $\mathbb{F}$ is the trivial $\mathbb{F}[α]/(α^n)$-module, cf.\ \cite[Appendix B.2]{Ma02}. This can be used to compute the canonical $\A_∞$-structure on the group cohomology $\Ext^*_{\fp {\rm C}_{m}}(\fp,\fp)$, where $m∈\Z_{\geq 1}$ and ${\rm C}_m$ is the cyclic group of order $m$, cf.\ \cite[Theorem 4.3.8]{Ve08}.

\name{Vejdemo-Johansson} developed algorithms for the computation of minimal models~\cite{Ve08}. He applied these algorithms to compute large enough parts of the canonical $\A_∞$-structures of the group cohomologies $\Ext^*_{\fs {\rm D}_8}(\fs,\fs)$ and $\Ext^*_{\fs {\rm D}_{16}}(\fs,\fs)$ to distinguish them, where ${\rm D}_8$ and ${\rm D}_{16}$ denote dihedral groups. He stated a conjecture on the complete $\A_∞$-structure on $\Ext^*_{\fs {\rm D}_8}(\fs,\fs)$. Furthermore, he computed parts of the canonical $\A_∞$-structure on $\Ext^*_{\fs {\rm Q}_8}(\fs, \fs)$ for the quaternion group ${\rm Q}_8$. He conjecturally stated the minimal complexity of such a structure. Based on this work, there are now built-in algorithms for the Magma computer algebra system. These are capable of computing partial $\A_∞$-structures on the group cohomology of $p$-groups.

In \cite{Ve082}, \name{Vejdemo-Johansson} examined  the canonical $\A_∞$-structure $(m_n)_{n\geq 1}$  on the group cohomology $\Ext^*_{\fp({\rm C}_k\times {\rm C}_l)}(\fp, \fp)$ of the abelian group ${\rm C}_k\times {\rm C}_l$ for $k,l\geq 4$ such that $k,l$ are multiples of $p$. He showed that for infinitely many $n∈\Z_{\geq 1}$, the operation $m_n$ is non-zero. 

In \cite{Kl10}, \name{Klamt} investigated canonical $\A_∞$-structures in the context of the representation theory of Lie-algebras. In particular, given certain direct sums $M$ of parabolic Verma modules, she examined the canonical $\A_∞$-structure $(m'_k)_{k\geq 1}$ on $\Ext^*_{\mathcal{O}^{\mathfrak{p}}}(M, M)$. She proved upper bounds for the maximal $k∈\Z_{\geq 1}$ such that $m'_k$ is non-vanishing and computed the complete $\A_∞$-structure in certain cases. 	


\paragraph{The result}

For $n∈\Z_{\geq 1}$, we denote by $\Sy_n$ the symmetric group on $n$ elements.
 
The group cohomology $\Ext^*_{\fp\!\Sp}(\fp,\fp)$ is well-known. For example, in \cite[p.~74]{Be98}, it is calculated using group cohomological methods. 

Here, we will construct the canonical $\A_∞$-structure on $\Ext^*_{\fp\!\Sp}(\fp, \fp)$. 

We obtain homogeneous elements $ι,χ∈\Hom_{\fp\!\Sp}^*(\pres \fp, \pres \fp)=:A$ of degree $|ι|=2(p-1)=: l$ and $|χ|=l-1$ such that $ι^j,χ\circ ι^j=:χι^j$ are cycles for all $j∈\Z_{\geq 0}$ and such that their set of homology classes $\{\overline{ι^j} \mid j∈\Z_{\geq 0}\}\sqcup \{\overline{χι^j} \mid j∈\Z_{\geq 0}\}$ is an $\fp$-basis of $\Ext^*_{\fp\!\Sp}(\fp, \fp)=\Hm^*A$, cf.\ \cref{pp:iota}.

For all primes $p$, the canonical $\A_∞$-structure $(m'_n:(\Hm^*A)^{\otimes n}→\Hm^*A)_{n\geq 1}$ on $\Hm^*A$ is given as follows.  

On the elements $\overline{χ^{a_1}ι^{j_1}}\otimes \cdots \otimes \overline{χ^{a_n}ι^{j_n}}$, $n∈\Z_{\geq 1}$, $a_i∈\{0,1\}$ and $j_i∈\Z_{\geq 0}$ for $i∈\{1,…,n\}$, the maps $m'_n$ are given as follows, cf.\ \cref{defall,bem:comp}.
 
If there is an $i∈\{1,…,n\}$ such that $a_i=0$, then
\begin{align*}
m'_n(\overline{χ^{a_1}ι^{j_1}}\otimes \cdots \otimes \overline{χ^{a_n}ι^{j_n}}) \eqs& 0 \hphantom{\overline{χ^{a_1 + a_1}ι^{j_1 + j_2}}} \text{for $n\neq 2$ and}\\
m'_2(\overline{χ^{a_1}ι^{j_1}}\otimes  \overline{χ^{a_2}ι^{j_2}}) \eqs & \overline{χ^{a_1 + a_1}ι^{j_1 + j_2}}.
\end{align*}
If all $a_i$ equal $1$, then 
\begin{align*}
m'_n(\overline{χι^{j_1}}\otimes \cdots \otimes \overline{χι^{j_n}}) \eqs&0  \hphantom{(-1)^p\overline{ι^{p-1+j_1+… + j_p}}}\text{for $n\neq p$ and }\\
m'_p(\overline{χι^{j_1}}\otimes \cdots \otimes \overline{χι^{j_p}}) \eqs  &
(-1)^p\overline{ι^{p-1+j_1+… + j_p}}. 
\end{align*}
In particular, we have $m'_n=0$ for all $n∈ \Z_{\geq 1}\setminus\{2,p\}$.



\subsection{Outline}
\paragraph{Section 1}
%
The goal of \cref{secpres} is to obtain a projective resolution of 
 the trivial $\fp\!\Sp$-Specht module $\fp$. A well-known method for that is "Walking around the Brauer tree", cf.\ \cite{Gr74}. 
Instead, we use locally integral methods to obtain a projective resolution in an explicit and straightforward manner.

%

Over $ℚ$, the Specht modules are absolutely simple. Therefore we have a morphism of $\zp$-algebras
%
%
$r:\zp\!\Sp→\prod_{λ\dashv p} \End_{\zp} S^λ_{\zp}=:Γ$ induced by the operation of the elements of $\zp\!\Sp$ on the Specht modules $S^λ$ for  partitions $λ$ of $p$, which becomes an Wedderburn isomorphism when tensoring with $ℚ$. So $Γ$ is a product of matrix rings over $\zp$.
There is a well-known description of  $\im r=:Λ$, 
which we use for $p\geq 3$ to obtain projective $Λ$-modules $\tilde P_k\subseteq Λ$, $k ∈[1,p-1]$, and to construct the indecomposible projective resolution $\pres\zp$ of the trivial $\zp\!\Sp$-Specht module $\zp$. The non-zero parts of $\pres\zp$ are periodic with period length $l=2(p-1)$. In \cref{sec:prfp}, we reduce $\pres\zp$ modulo $p$ to obtain a projective resolution $\pres\fp$ of the trivial $\fp\!\Sp$-Specht module $\fp$. 

\paragraph{Section 2}


The goal of \cref{secainf} is to compute a minimal model of the dg-algebra $\Hom_{\fp\!\Sp}^*(\pres\fp, \pres\fp)=: A$ by equipping its homology $\Ext^*_{\fp\!\Sp}(\fp,\fp) = \Hm^*A$ with a suitable $\A_∞$-structure and  finding a quasi-isomorphism of $\A_∞$-algebras from $\Hm^* A$ to $A$.

Towards that end, we recall the basic definitions concerning $\A_∞$-algebras and some general results in \cref{generaltheory}.


While there does not seem to be a substantial difference between the cases $p=2$ and $p\geq 3$, we separate them to  simplify notation and argumentation. 
Consider the case $p\geq 3$. In \cref{subsec:homology}, we obtain a set of cycles $\{ι^j \mid j∈\Z_{\geq 0}\}\cup\{χι^j \mid j∈\Z_{\geq 0}\}$ in $A$ such that their homology classes are a graded basis of $\Hm^*A$. 
In \cref{subsec:minmod}, we obtain a suitable $\A_∞$-structure on $\Hm^*A$ and a quasi-isomorphism of $\A_∞$-algebras from $\Hm^*A$ to $A$. For the prime $2$, both steps are combined in the short \cref{prime2}.



\subsection{Notations and conventions}
\label{sec:not}
\paragraph{Stipulations} 
\begin{itemize}
\item \textbf{For the remainder of this document, $p$ will be a prime with $p\geq 3$.}
\item \textbf{Write $l:=2(p-1)$.} This will give the period length of the constructed projective resolution of $\fp$ over $\fp\!\Sp$, cf.\ e.g.\ \eqref{omega}, \cref{lem:prfp}.
\end{itemize}
\paragraph{Miscellaneous}
\begin{itemize}
\item Concerning "$∞$", we assume the set $\Z\cup\{∞\}$ to be ordered in such a way that $∞$ is greater than any integer, i.e.\ $∞> z$ for all $z∈\Z$, and that the integers are ordered as usual.
\item For $a∈\Z$, $b∈\Z\cup \{∞\}$, we denote by $[a,b] := \{z∈\Z \mid a\leq z\leq b\} \subseteq \Z$ the integral interval. In particular, we have $[a,∞] = \{z∈\Z \mid  z\geq a\}\subseteq \Z$ for $a∈\Z$.
\item For $n∈\Z_{\geq 0}$, $k∈\Z$, let the binomial coefficient $\binom{n}{k}$ be defined by the number of subsets of the set $\{1,…,n\}$ that have cardinality $k$. In particular, if $k<0$ or $k>n$, we have $\binom{n}{k}= 0$. Then the formula $\binom{n}{k-1} + \binom{n}{k} = \binom{n+1}{k}$ holds for all $k∈\Z$.
\item For a commutative ring $R$, an $R$-module $M$ and $a,b∈M$, $c∈R$, we write
\begin{align*} b&\equiv_c a &:\Longleftrightarrow & & a-b ∈ cM.\end{align*}
Often we have $M=R$ as module over itself.

\item Modules are right-modules unless otherwise specified. 
%
%
\item For sets, we denote by $\sqcup$  the disjoint union of sets. 
\item $|\cdot|$: For a homogeneous element $x$ of a graded module or a graded map $g$ between graded modules, we denote by $|x|$ resp.\ $|g|$ their degrees (This is not unique for $x=0$ resp.\ $g=0$). For $y$ a real number, $|y|$ denotes its absolute value. 
\end{itemize}
\paragraph{Symmetric Groups}
Let $n∈\Z_{\geq 1}$. We denote the symmetric group von $n$ elements by $\Sy_n$.
For a partition $λ\dashv n$, we denote the corresponding Specht module by $S^λ$.
\paragraph{Complexes} Let $R$ be a commutative ring and $B$ an $R$-algebra.
\begin{itemize}
\item  For a complex of $B$-modules 
\[\cdots \rightarrow C_{k+1}\xrightarrow{d_{k+1}}C_k\xrightarrow{d_k} C_{k-1}\rightarrow\cdots\hphantom{,},\]
its $k$-th boundaries, cycles and homology groups are defined by $\text{B}^k:=\im d_{k+1}$, $\text{Z}^k:=\ker d_{k}$ and $\Hm^k:= \text{Z}^k/\text{B}^k$.

For a cycle $x∈\text{Z}^k$, we denote by $\overline{x}:=x+\text{B}^k∈\Hm^k$ its equivalence class in homology.
\item Let \begin{align*}
C\eqs&(\cdots \rightarrow C_{k+1}\xrightarrow{d_{k+1}}C_k\xrightarrow{d_k} C_{k-1}\rightarrow\cdots)\\
C'\eqs&(\cdots \rightarrow C'_{k+1}\xrightarrow{d'_{k+1}}C'_k\xrightarrow{d'_k} C'_{k-1}\rightarrow\cdots)
\end{align*}
be two complexes of $B$-modules.

Given $z∈\Z$, let 
\begin{align*}
\Hom_B^z(C,C') := \prod_{i∈\Z} \Hom_B(C_{i+z},C'_i).
\end{align*}
For an additional complex $C''=(\cdots \rightarrow C''_{k+1}\xrightarrow{d''_{k+1}}C''_k\xrightarrow{d''_k} C''_{k-1}\rightarrow\cdots)$ and maps $h=(h_i)_{i∈\Z}∈\Hom_B^m(C,C')$, $h'=(h'_i)_{i∈\Z}∈\Hom_B^n(C',C'')$, $m,n∈\Z$, we define the composition by component-wise composition as
\begin{align*}
h'\circ h := (h'_{i}\circ h_{i+n})_{i∈\Z} ∈ \Hom_B^{m+n}(C,C'').
\end{align*}

We will assemble elements of $\Hom_B^z(C,C')$ as sums of their non-zero components, which motivates the following notations regarding "extensions by zero" and sums.

For a map $g:C_x→C'_y$\,, we define $\ls g\rs_x^y∈\Hom_B^{x-y}(C,C')$  by
\begin{align*}
(\ls g\rs_x^y)_{i} := \begin{cases} g & \text{ for }i=y\\ 0 & \text{ for } i∈\Z\setminus \{y\} \end{cases}\, .
\end{align*}

Let $k∈\Z$. Let $I$ be a (possibly infinite) set. Let $g_i = (g_{i,j})_j∈\Hom_B^k(C,C')$ for $i∈I$  such that $\{i∈I \mid g_{i,j}\neq 0\}$ is finite for all $j∈\Z$.\\*
We define the sum $\sum_{i∈I} g_i∈\Hom_B^k(C,C')$ by
\begin{align*}
\left(\dum_{i∈I} g_i\right)_j := \sum_{i∈I,g_{i,j}\neq 0} g_{i,j}\,.
\end{align*}

The graded $R$-module $\Hom_B^*(C,C') := \bigoplus_{k∈\Z} \Hom_B^k(C,C')$ becomes a complex via the differential $d_{\Hom_B^*(C,C')}$, which is defined on elements $g∈\Hom_B^k(C,C')$, $k∈\Z$ by
\begin{align*}
d_{\Hom_B^*(C,C')}(g) := d' \circ g -(-1)^k g\circ d ∈\Hom_B^{k+1}(C,C'),
\end{align*}
where $d := (d_{i+1})_{i∈\Z}= \sum_{i∈\Z}\ls d_{i+1}\rs^{i}_{i+1}∈\Hom_B^1(C,C)$ and analogously $d' := (d'_{i+1})_{i∈\Z}= \sum_{i∈\Z}\ls d'_{i+1}\rs^{i}_{i+1}∈\Hom_B^1(C',C')$.

An element $h∈\Hom_B^0(C,C')$ is called a complex morphism if it satisfies $d_{\Hom_B^*(C,C')}(h) = 0$, i.e.\ $d'\circ g = g\circ d$. 
\end{itemize}

\section{The projective resolution of \texorpdfstring{$\fp$}{Fp} over \texorpdfstring{$\fp\!\Sp$}{FpSp}}
\label{secpres}
\subsection{A description of \texorpdfstring{$\zp\!\Sp$}{ZpSp}}
\kommentar{
\label{desczpsp}
In this paragraph, we review results found e.g.\ in \cite[Chapter 4.2]{Ku99}. We use the notation of \cite{Ja78}.

Let $n∈\Z_{\geq 1}$. 

A partition of the form $λ^k := (n-k+1,1^{k-1})$, $k∈[1,n]$ is called a \textit{hook partition} of $n$.

Suppose $λ\dashv n$, i.e.\ $λ$ is a partition of $n$.

Let $S^λ$ be the corresponding integral Specht module, which is a right $\Z\!\Sy_n$-module, cf. \cite[4.3]{Ja78}.
Then $S^{λ}$ is finitely generated free over $\Z$, cf. \cite[8.1, proof of 8.4]{Ja78}, having a standard $\Z$-basis consisting of the standard $λ$-polytabloids. We write $n_λ$ for the rank of $S^λ$.

For a tuple $b=(b_2,b_3,…,b_k)$, $k∈[1,n]$,  of pairwise distinct elements of $[1,n]$, let $\lpt b\rpt$ be the  $λ^k$-polytabloid generated by the $λ^k$-tabloid 
}
\tikzset{rightrule/.style={%
        execute at end cell={%
            \draw [line cap=rect,#1] (\tikzmatrixname-\the\pgfmatrixcurrentrow-\the\pgfmatrixcurrentcolumn.north east) -- (\tikzmatrixname-\the\pgfmatrixcurrentrow-\the\pgfmatrixcurrentcolumn.south east);%
        }
    },
    bottomrule/.style={%
        execute at end cell={%
            \draw [line cap=rect,#1] (\tikzmatrixname-\the\pgfmatrixcurrentrow-\the\pgfmatrixcurrentcolumn.south west) -- (\tikzmatrixname-\the\pgfmatrixcurrentrow-\the\pgfmatrixcurrentcolumn.south east);%
        }
    },
    toprule/.style={%
        execute at end cell={%
            \draw [line cap=rect,#1] (\tikzmatrixname-\the\pgfmatrixcurrentrow-\the\pgfmatrixcurrentcolumn.north west) -- (\tikzmatrixname-\the\pgfmatrixcurrentrow-\the\pgfmatrixcurrentcolumn.north east);%
        }
    }
}

Recall that $p\geq 3$ is a prime. 

For $R$ a ring and $λ\dashv p$ a partition of $p$, the $R\!\Sp$-Specht module $S^λ$ is finitely generatey free over $R$ with dimension independent of $R$, cf.\ \cite[8.1, proof of 8.4]{Ja78}. We denote this dimension by $n_λ$.

A partition of the form $λ^k := (p-k+1,1^{k-1})$, $k∈[1,p]$ is called a \textit{hook partition} of $p$.

Over the valuation ring $\zp$, there is a well-known description of the group algebra $\zp\!\Sp$,  cf. e.g.\ \cite[Corollary 4.2.8]{Ku99} (using \cite{Pe71}), cf.\ also \cite[Chapter 7]{Ro80}:
\begin{pp}
\label[pp]{lambda}
Set 
$n^{k}_{\rm b} = \binom{p-2}{k-1}$, $n^{k}_{\rm c} = \binom{p-2}{k-2}$. Then $n^{k}_{\rm b} + n^{k}_{\rm c} = \binom{p-1}{k-1} = n_{λ^k}$. Set 
$ Γ := \prod_{λ \dashv p}  \zp^{n_λ\times n_λ} $.
For $ρ∈Γ$, and $λ\dashv p$, we denote by $ρ^λ$ the $λ$-th component of $ρ$. For $λ=λ^k$, $k∈[1,p]$, a hook partition, we name certain subblocks of $ρ^{λ^k}$ as follows.
\begin{equation*}
ρ^{λ^k} = 
\xyprot{
\begin{tikzpicture}[decoration=brace,baseline]
    \matrix (m) [matrix of math nodes,left delimiter=(,right delimiter=),
      nodes={
outer sep=0pt},
      row 1/.style={bottomrule},column 1/.style=rightrule] {
        ρ^{λ^k}_{\rm cc} & ρ^{λ^k}_{\rm bc} \\
        ρ^{λ^k}_{\rm cb} & ρ^{λ^k}_{\rm bb} \\
    };
    \draw[decorate,transform canvas={xshift=1.5em},thick] (m-1-2.north east) -- node[right=2pt] {$n^k_{\rm c}$ rows} (m-1-2.south east);
    \draw[decorate,transform canvas={xshift=1.5em},thick] (m-2-2.north east) -- node[right=2pt] (ru) {$n^k_{\rm b}$ rows} (m-2-2.south east);
    \draw[decorate,transform canvas={yshift=0.5em},thick] (m-1-1.north west) -- node[above=2pt] {$n^k_{\rm c}$,} (m-1-1.north east);
    \draw[decorate,transform canvas={yshift=0.5em},thick] (m-1-2.north west) -- node[above=2pt] {$n^k_{\rm b}$ \rlap{columns}} (m-1-2.north east);
    
    \node[transform canvas={xshift=1.5em}] at (ru.south east) {.};
\end{tikzpicture}
}
\end{equation*}
We have the following $\zp$-subalgebra $Λ$ of $Γ$.
\begin{align*}
Λ &:= \{ρ∈Γ \mid  ρ^{λ^k}_{\rm bb} \equiv_p ρ^{λ^{k+1}}_{\rm cc}\text{ for }k∈[1,p-1]\text{ and } ρ^{λ^k}_{\rm bc} \equiv_p 0\text{ for }k∈[1,p]\}
\end{align*}
Now there is an isomorphism of $\zp$-algebras
\vspace{-1ex} 
\[\begin{array}{rccc}
r:&\zp\!\Sp & \xrightarrow{\phantom{X}\sim\phantom{X}}& Λ.
\end{array}
\]
such that $ρ∈Λ$ acts on the trivial $\zp\!\Sp$-module $\zp$ by multiplication with its ($1\times 1$ / scalar-)component $ρ^{λ^1}$, i.e.\ for $x∈\zp\!\Sp$ and for $y∈\zp$, we have $yx = y\cdot r(x)^{λ^1}$.
%

\end{pp}

\begin{bsp}
For $p=5$, the $\zp$-algebra $\Z_{(5)}\Sy_5$ is isomorphic to the subalgebra $Λ$ of $Γ=\Z_{(5)}^{1\times 1}\times \Z_{(5)}^{4\times 4}\times \Z_{(5)}^{6\times 6}\times \Z_{(5)}^{4\times 4}\times \Z_{(5)}^{1\times 1}\times  \Z_{(5)}^{5\times 5}\times \Z_{(5)}^{5\times 5}$ described as
\begin{center}
\xyprot{
\noindent
\begingroup
\renewcommand*{\arraystretch}{0.25}
\newcommand{\ofstr}{-0.9mm}
\begin{tikzpicture}[inner xsep=0.1mm, inner ysep=0.0mm,node distance=4mm,
space/.style={minimum height=1.0em,matrix of nodes,row sep=1.0mm,column sep=1mm,
}
]
\matrix (m1) [space, below delimiter=\}\rotatebox{0}{\smash{$\begin{matrix*}[l]\scriptstyle{\times}\\ \scriptstyle\times \\ \scriptstyle\times  \\ \scriptstyle\times \\ \scriptstyle{\times}\end{matrix*}$}}]
{
\asm{\,\Z_{(5)}}\\
};
\matrix (m2) [right=of m1, space, below delimiter=\}\rotatebox{0}{$\begin{matrix*}[l]\scriptstyle\times \\ \scriptstyle\times  \\ \scriptstyle\times \\ \scriptstyle{\times}{\times}\end{matrix*}$}]
{
\asm{\Z_{(5)}} & \asm{5\Z_{(5)}} & \asm{5\Z_{(5)}} & \asm{5\Z_{(5)}} \\
\asm{\Z_{(5)}} & \asm{\Z_{(5)}} & \asm{\Z_{(5)}} & \,\asm{\Z_{(5)}}\, \\
\asm{\Z_{(5)}} & \asm{\Z_{(5)}} & \asm{\Z_{(5)}} & \asm{\Z_{(5)}} \\
\asm{\Z_{(5)}} & \asm{\Z_{(5)}} & \asm{\Z_{(5)}} & \asm{\Z_{(5)}} \\
};
\matrix (m3) [right=of m2, space, below delimiter=\}\rotatebox{0}{$\begin{matrix*}[l]\scriptstyle\times  \\ \scriptstyle\times \\ \scriptstyle{\times}{\times}{\times}\end{matrix*}$}]
{
\asm{\Z_{(5)}} & \asm{\Z_{(5)}} & \asm{\Z_{(5)}} & \asm{5\Z_{(5)}} & \asm{5\Z_{(5)}} & \asm{5\Z_{(5)}} \\
\asm{\Z_{(5)}} & \asm{\Z_{(5)}} & \asm{\Z_{(5)}} & \asm{5\Z_{(5)}} & \asm{5\Z_{(5)}} & \asm{5\Z_{(5)}} \\
\asm{\Z_{(5)}} & \asm{\Z_{(5)}} & \asm{\Z_{(5)}} & \asm{5\Z_{(5)}} & \asm{5\Z_{(5)}} & \asm{5\Z_{(5)}} \\
\asm{\Z_{(5)}} & \asm{\Z_{(5)}} & \asm{\Z_{(5)}} & \asm{\Z_{(5)}} & \asm{\Z_{(5)}} & \asm{\Z_{(5)}} \\
\asm{\Z_{(5)}} & \asm{\Z_{(5)}} & \asm{\Z_{(5)}} & \asm{\Z_{(5)}} & \asm{\Z_{(5)}} & \asm{\Z_{(5)}} \\
\asm{\Z_{(5)}} & \asm{\Z_{(5)}} & \asm{\Z_{(5)}} & \asm{\Z_{(5)}} & \asm{\Z_{(5)}} & \asm{\Z_{(5)}} \\
};
\matrix (m4) [right=of m3, space, below delimiter=\}\rotatebox{0}{$\begin{matrix*}[l]\scriptstyle\times  \\  \scriptstyle{\times}{\times}{\times}{\times}\end{matrix*}$}]
{
\asm{\Z_{(5)}} & \asm{\Z_{(5)}} & \asm{\Z_{(5)}} & \asm{5\Z_{(5)}} \\
\asm{\Z_{(5)}} & \asm{\Z_{(5)}} & \asm{\Z_{(5)}} & \asm{5\Z_{(5)}} \\
\asm{\Z_{(5)}} & \asm{\Z_{(5)}} & \asm{\Z_{(5)}} & \asm{5\Z_{(5)}} \\
\asm{\Z_{(5)}} & \asm{\Z_{(5)}} & \asm{\Z_{(5)}} & \asm{\Z_{(5)}} \\
};
\matrix (m5) [right=of m4,space, below delimiter=\}\rotatebox{0}{$\begin{matrix*}[l] \scriptstyle{\times}{\times}{\times}{\times}{\times}\end{matrix*}$}]
{
\asm{\Z_{(5)}\,}\\
};
\node[draw=black, fit=(m1-1-1)](B12) {};
\node[draw=black, fit=(m2-1-1)](B21) {};
\node[draw=black, fit=(m2-2-2.north west) (m2-4-4.south east)](B22) {};
\node[draw=black, fit=(m3-1-1.north west) (m3-3-3.south east)](B31) {};
\node[draw=black, fit=(m3-4-4.north west) (m3-6-6.south east)](B32) {};
\node[draw=black, fit=(m4-1-1.north west) (m4-3-3.south east)](B41) {};
\node[draw=black, fit=(m4-4-4)] (B42) {};
\node[draw=black, fit=(m5-1-1)] (B51) {};
\draw[black] (B12.east) --  node[sloped,above,pos=0.6]{\raisebox{0.5mm}{$\scriptscriptstyle 5$}} (B21.west);
\draw[black] (B22.east) --  node[sloped,above,pos=0.6]{\raisebox{0.5mm}{$\scriptscriptstyle 5$}} (B31.west);
\draw[black] (B32.east) --  node[sloped,above,pos=0.6]{\raisebox{0.5mm}{$\scriptscriptstyle 5$}} (B41.west);
\draw[black] (B42.east) --  node[sloped,above,pos=0.6]{\raisebox{0.5mm}{$\scriptscriptstyle 5$}} (B51.west);
%
\node[left=\ofstr] (C011) at (m1-1-1.south west) {};
\node[below=\ofstr] (C11) at (C011) {};
\node[below=\ofstr] (C12) at (m1-1-1.south east) {};
\node[below=\ofstr] (C13) at (m2-1-1.south west) {};
\node[below=\ofstr] (C14) at (m2-1-4.south east) {};
\node[above=\ofstr] (C15) at (m2-1-4.north east) {};
\node[above=\ofstr] (C16) at (m2-1-1.north west) {};
\node[above=\ofstr] (C17) at (m1-1-1.north east) {};
\node[left=\ofstr] (C018) at (m1-1-1.north west) {};
\node[above=\ofstr] (C18) at (C018) {};
\draw[red,rounded corners=2pt] (C11.center) -- (C12.center) -- (C13.center) -- (C14.center) -- (C15.center) --
node[above]{\raisebox{1.5mm}{$\tilde P_1$}} (C16.center)  -- (C17.center) -- (C18.center) -- cycle;
\node[below=\ofstr] (C21) at (m2-2-1.south west) {};
\node[below=\ofstr] (C22) at (m2-2-4.south east) {};
\node[below=\ofstr] (C23) at (m3-1-1.south west) {};
\node[below=\ofstr] (C24) at (m3-1-6.south east) {};
\node[above=\ofstr] (C25) at (m3-1-6.north east) {};
\node[above=\ofstr] (C26) at (m3-1-1.north west) {};
\node[above=\ofstr] (C27) at (m2-2-4.north east) {};
\node[above=\ofstr] (C28) at (m2-2-1.north west) {};
\draw[red,rounded corners=2pt] (C21.center) -- (C22.center) -- (C23.center) -- (C24.center) -- (C25.center) -- node[above]{\raisebox{1.5mm}{$\tilde P_2$}} (C26.center)  -- (C27.center) -- (C28.center) -- cycle;
\node[below=\ofstr] (C31) at (m3-4-1.south west) {};
\node[below=\ofstr] (C32) at (m3-4-6.south east) {};
\node[below=\ofstr] (C33) at (m4-1-1.south west) {};
\node[below=\ofstr] (C34) at (m4-1-4.south east) {};
\node[above=\ofstr] (C35) at (m4-1-4.north east) {};
\node[above=\ofstr] (C36) at (m4-1-1.north west) {};
\node[above=\ofstr] (C37) at (m3-4-6.north east) {};
\node[above=\ofstr] (C38) at (m3-4-1.north west) {};
\draw[red,rounded corners=2pt] (C31.center) -- (C32.center) -- (C33.center) -- (C34.center) -- (C35.center) -- node[above]{\raisebox{1.5mm}{$\tilde P_3$}} (C36.center)  -- (C37.center) -- (C38.center) -- cycle;
\node[below=\ofstr] (C41) at (m4-4-1.south west) {};
\node[below=\ofstr] (C42) at (m4-4-4.south east) {};
\node[below=\ofstr] (C43) at (m5-1-1.south west) {};
\node[right=\ofstr](C044) at (m5-1-1.south east) {};
\node[below=\ofstr] (C44) at (C044) {};
\node[right=\ofstr](C045) at (m5-1-1.north east) {};
\node[above=\ofstr] (C45) at (C045) {};
\node[above=\ofstr] (C46) at (m5-1-1.north west) {};
\node[above=\ofstr] (C47) at (m4-4-4.north east) {};
\node[above=\ofstr] (C48) at (m4-4-1.north west) {};
\draw[red,rounded corners=2pt] (C41.center) -- (C42.center) -- (C43.center) -- (C44.center) -- (C45.center) --
node[above]{\raisebox{1.5mm}{$\tilde P_4$}} (C46.center)  -- (C47.center) -- (C48.center) -- cycle;
%
\end{tikzpicture} \\*
\nopagebreak
\begin{tikzpicture}[inner xsep=0.1mm, inner ysep=0.0mm,node distance=4mm,
cell/.style={rectangle,draw=black},
space/.style={minimum height=1.0em,matrix of nodes,row sep=0.0mm,column sep=1mm,
}]
\matrix (mx) [space, below delimiter=\}\rotatebox{0}{$\begin{matrix*}[l] \scriptstyle{\times}{\times} \\ \scriptstyle{\times}{\times}{\times}  \end{matrix*}$}]
{
\asm{\Z_{(5)}} & \asm{\Z_{(5)}} & \asm{\Z_{(5)}} & \asm{\Z_{(5)}} & \asm{\Z_{(5)}}\\
\asm{\Z_{(5)}} & \asm{\Z_{(5)}} & \asm{\Z_{(5)}} & \asm{\Z_{(5)}} & \asm{\Z_{(5)}}\\
\asm{\Z_{(5)}} & \asm{\Z_{(5)}} & \asm{\Z_{(5)}} & \asm{\Z_{(5)}} & \asm{\Z_{(5)}}\\
\asm{\Z_{(5)}} & \asm{\Z_{(5)}} & \asm{\Z_{(5)}} & \asm{\Z_{(5)}} & \asm{\Z_{(5)}}\\
\asm{\Z_{(5)}} & \asm{\Z_{(5)}} & \asm{\Z_{(5)}} & \asm{\Z_{(5)}} & \asm{\Z_{(5)}}\\
};
\matrix (my) [space,right=of mx, below delimiter=\}\rotatebox{0}{$\begin{matrix*}[l]\scriptstyle{\times}\\ \scriptstyle{\times}{\times} \\ \scriptstyle{\times}{\times}  \end{matrix*}$}]
{
\asm{\Z_{(5)}} & \asm{\Z_{(5)}} & \asm{\Z_{(5)}} & \asm{\Z_{(5)}} & \asm{\Z_{(5)}}\\
\asm{\Z_{(5)}} & \asm{\Z_{(5)}} & \asm{\Z_{(5)}} & \asm{\Z_{(5)}} & \asm{\Z_{(5)}}\\
\asm{\Z_{(5)}} & \asm{\Z_{(5)}} & \asm{\Z_{(5)}} & \asm{\Z_{(5)}} & \asm{\Z_{(5)}}\\
\asm{\Z_{(5)}} & \asm{\Z_{(5)}} & \asm{\Z_{(5)}} & \asm{\Z_{(5)}} & \asm{\Z_{(5)}}\\
\asm{\Z_{(5)}} & \asm{\Z_{(5)}} & \asm{\Z_{(5)}} & \asm{\Z_{(5)}} & \asm{\Z_{(5)}}\\
};
\end{tikzpicture}.
\endgroup
}
\end{center}
An entry in this tuple of matrices indicates that an element of $Λ$ must have its corresponding entry in the indicated set. A relation "\begin{tikzpicture}  \node (n1) {}; \node (n2) [right=of n1] {}; 
\draw[black] (n1.east) --  node[sloped,above]{$\scriptscriptstyle 5$} (n2.west);
\end{tikzpicture}" between (equal sized) subblocks indicates that these subblocks  are equivalent modulo $5$, i.e.\ the difference of corresponding entries is an element of $5\Z_{(5)}$. The blocks are labeled with the diagrams of the corresponding partitions.
The right ideals $\tilde P_i = \tilde e_iΛ$, $i∈[1,4]=[1,p-1]$ (cf.\ the definitions below) are framed with red lines.
\end{bsp}
\subsection
{A projective resolution of \texorpdfstring{$\zp$}{Zp} over \texorpdfstring{$\zp\!\Sp$}{ZpSp}}
\label{secpreszp}

Recall that $p\geq 3$ is a prime.

Recall from \cref{lambda} that $Λ$ is a subring of $Γ=\prod_{λ\dashv p} \zp^{n_λ\times n_λ}$.

For $λ\dashv p$ and $i,j∈[1,n_λ]$, we set $η_{λ,i,j}$ to be the element of $Γ$ such that $(η_{λ,i,j})^{\tilde{λ}} = 0$ for $\tilde{λ}\neq λ$ and $(η_{λ,i,j})^{λ}∈\Z^{n_λ\times n_λ}$ has entry $1$ at position $(i,j)$ and zeros elsewhere.

Let $k∈[1,p-1]$. We obtain the idempotent
\begin{align*}
\tilde e_k :=  η_{λ^k,n_{\rm c}^{k}+1,n_{\rm c}^{k}+1} + η_{λ^{k+1},1,1} ∈ Λ.
\end{align*}
We define corresponding projective right $Λ$-modules 
\begin{align*}
\tilde P_k :\eqs& \tilde e_k Λ \quad \text{ for $k∈[1,p-1]$.}
\end{align*}

\begin{bem}
\label[bem]{rem:moralg}
Let $A$ be an $R$-algebra and let $e,e'∈A$ be two idempotents. For the right modules $eA$, $e'A$, we have the isomorphism of $R$-Modules
\[ \begin{array}{rcl}
 \Hom_A(eA, e'A) &\underset{\sim}{\overset{T_{e',e}}{\longrightarrow}}& e'Ae\\
f &\longmapsto& T_{e',e}(f) := f(e)\\
T_{e',e}^{-1}(e'be):=(ea \mapsto e'bea) &\reflectbox{$\longmapsto$}&
 e'be \phantom{:= f(e)X}.
\end{array}\]
Thus given $m∈e'Ae$, the morphism $T_{e',e}^{-1}(m)$ acts on elements $x∈eA$ by the multiplication of $m$  on the left: $(T_{e',e}^{-1}(m))(x) = m\cdot x$.

Given idempotents $e,e',e''∈A$, and elements $f∈\Hom_A(eA,e'A)$, $g∈A(e'A,e''A)$, we have $T_{e'',e}(g\circ f) = g(f(e)) = g(e'f(e)) = g(e')\cdot f(e)= T_{e'',e'}(g)\cdot T_{e',e}(f)$.
\end{bem}
\begin{Def}
\label[Def]{def:diffs}
We define via \cref{rem:moralg}
\begin{equation*}
\begin{array}{lclll}
\hat e_{k}   &:\eqs& T_{\tilde e_k, \tilde e_k}^{-1}(\tilde e_k)				&∈\Hom_Λ(\tilde P_k,\tilde P_k) & \text{for }k∈[1,p-1]\\
\hat e_{1,1}    &:\eqs& T_{\tilde e_1,\tilde e_1}^{-1}(pη_{λ^1,1,1})			&∈\Hom_Λ(\tilde P_1,\tilde P_1) \\
\hat e_{p-1,p-1}&:\eqs& T_{\tilde e_{p-1},\tilde e_{p-1}}^{-1}(pη_{λ^p,1,1})			&∈\Hom_Λ(\tilde P_{p-1},\tilde P_{p-1})\\
\hat e_{k+1,k}  &:\eqs& T_{\tilde e_{k+1},\tilde e_k}^{-1}(η_{λ^{k+1},n_{\rm c}^{k+1}+1,1})	&∈\Hom_Λ(\tilde P_k,\tilde P_{k+1}) & \text{for }k∈[1,p-2]\\
\hat e_{k,k+1}  &:\eqs& T_{\tilde e_k,\tilde e_{k+1}}^{-1}(pη_{λ^{k+1},1,n_{\rm c}^{k+1}+1})	&∈\Hom_Λ(\tilde P_{k+1},\tilde P_k) & \text{for }k∈[1,p-2].
\end{array}
\end{equation*}
Note that $\hat e_k$ is the identity map on $\tilde P_k$ for $k∈[1,p-1]$.

Moreover, we define the $\zp\!\Sp$-linear map $\hat{ε}:\tilde P_1 → \zp$, $\hat{ε}(ρ) := ρ^{λ^1}$.
\end{Def}

It is straightforward to show the following lemma.
\begin{lemma}
\label[lemma]{lem:zprel}
We have
\begin{equation*}
\begin{array}{lclr}
\hat e_{1,1} + \hat e_{1,2}\circ \hat e_{2,1} &=& p\hat e_1\\
\hat e_{k,k-1}\circ\hat e_{k-1,k} + \hat e_{k,k+1}\circ\hat e_{k+1,k} &=& p\hat e_k &\text{ for }k∈[2,p-2]\\
\hat e_{p-1,p-2}\circ\hat e_{p-2,p-1} + \hat e_{p-1,p-1} &=& p\hat e_{p-1}\\
 \hat{ε}\circ \hat e_{1,1} &=& p\hat{ε}.
\end{array}
\end{equation*}
\end{lemma}

Furthermore, it is straightforward to check that we obtain a projective resolution of $\zp$ as follows. 
We set 
\begin{align*}
\tilde{\pr}_i :\eqs& \begin{cases} \tilde P_{ω(i)} & i\geq 0 \\ 0 & i<0 \end{cases},
\end{align*}
where the integer $ω(i)$ is given by the following construction: Recall the stipulation $l:=2(p-1)$. We have $i = jl+r$ for some $j∈\Z$ and $0\leq r\leq l-1$. Then
\begin{align}
\label{omega}
ω(i) := \begin{cases} r+1 & \text{ for }0\leq r \leq p-2\\ l-r = 2(p-1) -r & \text{ for } p-1 \leq r \leq 2(p-1)-1 = l-1 \end{cases}.
\end{align}
So $ω(i)$ increases by steps of one from $1$ to $p-1$ as $i$ runs from $jl$ to $jl+(p-2)$ and $ω(i)$ decreases from $p-1$ to $1$ as $i$ runs from $jl+(p-1)$ to $jl + (l-1)$. Finally we set
\begin{align*}
\hat d_i := \begin{cases} \hat e_{ω(i-1),ω(i)}: \tilde P_{ω(i)} → \tilde P_{ω(i-1)} & i\geq 1\\ 0 & i\leq 0 \end{cases}.
\end{align*} 
Now we have the projective resolution of $\zp$ 
\begin{align}
\label{presform1}
\cdots\xrightarrow{\hat d_3} \tilde{\pr}_2\xrightarrow{\hat d_2}\tilde{\pr}_1\xrightarrow{\hat d_1}\tilde{\pr}_0\xrightarrow{0=\hat d_0} 0\rightarrow\cdots\,,
\end{align}
written more explicitly as
\begin{align*}
\cdots\rightarrow &\tilde P_2\xrightarrow{\hat e_{1,2}} \tilde P_1 \xrightarrow{\hat e_{1,1}} \tilde P_1 \xrightarrow{\hat e_{2,1}} \tilde P_2 \rightarrow … 
\rightarrow \tilde P_{p-2} \xrightarrow{\hat e_{p-1,p-2}}\tilde P_{p-1} \nonumber \\
 &\xrightarrow{\hat e_{p-1,p-1}} \tilde P_{p-1} \xrightarrow{\hat e_{p-2,p-1}}\tilde P_{p-2}\rightarrow … \rightarrow \tilde P_2\xrightarrow{\hat e_{1,2}} \tilde P_1  \rightarrow 0\rightarrow \cdots\,,
\end{align*}
with augmentation $\hat{ε}:\tilde P_1 → \zp$.


\subsection{A projective resolution of \texorpdfstring{$\fp$}{Fp} over \texorpdfstring{$\fp\!\Sp$}{FpSp}} 
\label{sec:prfp}

The isomorphism $r:\zp\!\Sp \rightarrow Λ$ from \cref{lambda} induces an isomorphism of $\fp$-algebras $\fp\!\Sp = \zp\!\Sp/(p\zp\!\Sp) \xrightarrow{\bar{r}} Λ/(pΛ)=:\bar{Λ}$. \\*
For the sake of simplicity in the next step, we identify $\bar{Λ}$ and $\fp\!\Sp$ along $\bar{r}$.
\begin{lemma}
\label[lemma]{lem:prfp}
Recall that $p\geq 3$ is a prime. Applying the functor $-\underset{Λ}{\otimes} \leftidx{_Λ}{\bar{Λ}}{_{\bar{Λ}}}$\,, we obtain
\begin{itemize}
\item the projective modules $P_k:= \tilde P_k\underset{Λ}{\otimes} \leftidx{_Λ}{\bar{Λ}}{_{\bar{Λ}}}$ 
 for $k∈[1,p-1]$, 
\item $\fp:=\zp\underset{Λ}{\otimes} \leftidx{_Λ}{\bar{Λ}}{_{\bar{Λ}}}$ (the trivial $\fp\!\Sp$-module),
\item 
%
$
\begin{array}[t]{lclll}
e_k&:=&\hat e_k\underset{Λ}{\otimes} \leftidx{_Λ}{\bar{Λ}}{_{\bar{Λ}}}&∈\Hom_{\fp\Sp}(P_k,P_k)&\text{for $k∈[1,p-1]$,} \\
e_{1,1}&:=& \hat e_{1,1}\underset{Λ}{\otimes} \leftidx{_Λ}{\bar{Λ}}{_{\bar{Λ}}}&∈\Hom_{\fp\Sp}(P_1,P_1),\\
e_{p-1,p-1} &:=& \hat e_{p-1,p-1}\underset{Λ}{\otimes} \leftidx{_Λ}{\bar{Λ}}{_{\bar{Λ}}}&∈\Hom_{\fp\Sp}(P_{p-1},P_{p-1}), \\
e_{k+1,k}&:=& \hat e_{k+1,k}\underset{Λ}{\otimes} \leftidx{_Λ}{\bar{Λ}}{_{\bar{Λ}}}&∈\Hom_{\fp\Sp}(P_k,P_{k+1})&\text{for $k∈[1,p-2]$,} \\
e_{k,k+1} &:=& \hat e_{k,k+1}\underset{Λ}{\otimes} \leftidx{_Λ}{\bar{Λ}}{_{\bar{Λ}}}&∈\Hom_{\fp\Sp}(P_{k+1},P_k) &\text{for $k∈[1,p-2]$,} \\
ε &:=& \hat{ε}\underset{Λ}{\otimes} \leftidx{_Λ}{\bar{Λ}}{_{\bar{Λ}}}&∈\Hom_{\fp\Sp}(P_1,\fp) .
\end{array}
$
\end{itemize}
The complex 
\begin{align}
\label{prfp}
\pres \fp := (\pres \zp)\underset{Λ}{\otimes} \leftidx{_Λ}{\bar{Λ}}{_{\bar{Λ}}}= (\cdots\xrightarrow{d_3} {\pr}_2\xrightarrow{ d_2}{\pr}_1\xrightarrow{ d_1}{\pr}_0\xrightarrow{0=d_0} 0\rightarrow\cdots),
\end{align}
\begin{align*}
{\pr}_i :\eqs& \begin{cases} P_{ω(i)} & i\geq 0 \\ 0 & i<0 \end{cases} &
d_i :\eqs& \begin{cases} e_{ω(i-1),ω(i)}: P_{ω(i)} →P_{ω(i-1)} & i\geq 1\\ 0 & i\leq 0, \end{cases}
\end{align*}  
is a projective resolution of $\fp$ with augmentation $ε:P_1→\fp$.
More explicitly, $\pres \fp$ is
\begin{align*}
…\rightarrow & \underbrace{P_2}_{\mathclap{l+1}}\xrightarrow{ e_{1,2}}  \underbrace{P_1}_{\mathclap{l=2(p-1)}} \xrightarrow{ e_{1,1}}  \underbrace{P_1}_{\mathclap{(p-2)+p-1}} \xrightarrow{ e_{2,1}}  \underbrace{P_2}_{\mathclap{(p-2)+p-2}} \rightarrow … 
\rightarrow  \underbrace{P_{p-2}}_{\mathclap{p=(p-2)+2}} \xrightarrow{ e_{p-1,p-2}} \underbrace{P_{p-1}}_{\mathclap{(p-2)+1}} \nonumber \\
 &\xrightarrow{ e_{p-1,p-1}}  \underbrace{P_{p-1}}_{\mathclap{p-2}} \xrightarrow{ e_{p-2,p-1}} \underbrace{P_{p-2}}_{p-3}\rightarrow … \rightarrow  \underbrace{P_2}_{1} \xrightarrow{ e_{1,2}}  \underbrace{P_1}_0  \rightarrow 0.
\end{align*}
\end{lemma}
\begin{lemma}
\label[lemma]{lem:relfp}
Recall that $p\geq 3$ is a prime.
\begin{enumerate}[\rm (a)]
\item 
We have the relations
\begin{equation*}
\begin{array}{lclr}
e_{1,1} +  e_{1,2} \circ e_{2,1} &=& 0\\
 e_{k,k-1}\circ e_{k-1,k} +  e_{k,k+1}\circ e_{k+1,k} &=& 0 &\text{ for } k∈[2,p-2]\\
 e_{p-1,p-2} \circ e_{p-2,p-1} +  e_{p-1,p-1} &=& 0\\
 ε \circ e_{1,1} &=& 0
\end{array}
\end{equation*}
and $e_k$ is the identity on $P_k$ for $k∈[1,p-1]$.
\item Given $k∈[2,p-1]$, we have $\Hom_{\fp\!\Sp}(P_k,\fp) = \{0\}$.
\item Given $k,k'∈[1,p-1]$ such that $|k-k'|>1$, we have  $\Hom_{\fp\!\Sp}(P_k,P_{k'}) = \{0\}$.
\item The set $\{ε\}$ is an $\fp$-basis of $\Hom_{\fp\!\Sp}(P_1,\fp)$.
\end{enumerate}
\end{lemma}
Assertion (a) results from \cref{lem:zprel}.

Assertions (b), (c) and (d) are derived from corresponding assertions over $\zp\!\Sp$ using $\Hom_{\fp\!\Sp}(P/pP, M/pM) \simeq \Hom_{\zp\!\Sp}(P,M)/p\Hom_{\zp\!\Sp}(P,M)$ for $\zp\!\Sp$-modules $P$ and $M$, where $P$ is projective.

\section{\texorpdfstring{$\A_\infty$}{A(oo)}-algebras}
\label{secainf}
\subsection{Definitions, General theory}
\label{generaltheory}
In this subsection, we review results presented in \cite{Ke01} and we fix notation.

Let $R$ be a commutative  ring. 
We understand linear maps between $R$-modules to be $R$-linear. Tensor products are tensor products over $R$. By graded $R$-modules we understand $\Z$-graded $R$-modules.

\begin{Def}
In the definition of the tensor product of graded maps, we implement the \textit{Koszul sign rule}: Let $A_1,A_2,B_1,B_2$ be graded $R$-modules and $g:A_1→B_1$, $h:A_2→B_2$ graded maps. Then we set for homogeneous elements $x∈A_1,y∈A_2$
\begin{align}
\label{koszulraw}
(g\otimes h)(x\otimes y) := (-1)^{|h|\cdot|x|} g(x)\otimes h(y).
\end{align}
\end{Def}


Concerning the signs in the definition of $\A_∞$-algebras and $\A_∞$-morphisms, we follow the variant given e.g.\ in \cite{Le03} and \cite{Ka80}. 
%
\begin{Def}
Let $n∈\Z_{\geq 0}\cup \{∞\}$. 
\begin{enumerate}[(i)]
\item Let $A$ be a graded $R$-module. A \textit{pre-$\A_n$-structure on $A$} is a family of graded maps $(m_k:A^{\otimes k}→A)_{k∈[1,n]}$ with $|m_k|=2-k$ for $k∈[1,n]$. The tuple $(A,(m_k)_{k∈[1,n]})$ is called a pre-$\A_n$-algebra.
\item Let $A$, $A'$ be graded $R$-modules. A \textit{pre-$\A_n$-morphism from $A'$ to $A$} is a family of graded maps $(f_k:A'^{\otimes k}→A)_{k∈[1,n]}$ with $|f_k|=1-k$ for $k∈[1,n]$.
\end{enumerate}
\end{Def}
%
\begin{Def}
Let $n∈\Z_{\geq 0}\cup\{∞\}$. 
\begin{enumerate}[(i)]
\item 
An \textit{$\A_n$-algebra} is a pre-$\A_n$-algebra $(A,(m_k)_{k∈[1,n]})$ such that for \mbox{$k∈[1,n]$} 
\refstepcounter{equation}\label{ainfrel}
\begin{align}
\sum_{\natop{k=r+s+t,}{r,t\geq 0, s\geq 1}} (-1)^{rs+t}m_{r+1+t}\circ (1^{\otimes r}\otimes m_s\otimes 1^{\otimes t})=0 .
\tag*{(\theequation)[$k$]}
\end{align}
In abuse of notation, we sometimes abbreviate $A = (A,(m_k)_{k\geq 1})$ for $\A_∞$-algebras.
\item 
Let $(A',(m'_k)_{k∈[1,n]})$ and $(A,(m_k)_{k∈[1,n]})$ be $\A_n$-algebras. An \textit{$\A_n$-morphism} or \textit{morphism of $\A_n$-algebras} from $(A',(m'_k)_{k∈[1,n]})$ to $(A,(m_k)_{k∈[1,n]})$ is a pre-$\A_n$-morphism $(f_k)_{k∈[1,n]}$ such that for $k∈[1,n]$, we have 
\refstepcounter{equation}\label{finfrel}
\begin{align}
\sum_{\mathclap{\substack{k=r+s+t\\r,t\geq 0,s\geq 1}}} (-1)^{rs+t} f_{r+1+t}\circ (1^{\otimes r}\otimes m'_s\otimes 1^{\otimes t}) = \sum_{\mathclap{\substack{1\leq r\leq k\\ i_1+…+i_r=k\\ i_s\geq 1}}} (-1)^v m_r\circ (f_{i_1}\otimes f_{i_2}\otimes … \otimes f_{i_r}),
\tag*{(\theequation)[$k$]}
\end{align} 
where 
$v :=  
\sum_{1\leq t < s\leq r}(1-i_s)i_t$.
\end{enumerate}
\end{Def}

\begin{bsp}[dg-algebras]
\label[bsp]{ex:dg} Let $(A,(m_k)_{k\geq 1})$ be an $\A_\infty$-algebra.  
If $m_n=0$ for $n\geq 3$ then $A$ is called  a \textit{differential graded algebra} or \textit{dg-algebra}. In this case the equations \eqref{ainfrel}[$n$] for $n\geq 4$ become trivial: We have $(r+1+t)+s = n+1$ ⇒ $(r+1+t)+s\geq 5$ ⇒ $m_{r+1+t}=0$ or $m_s=0$. So all summands in \eqref{ainfrel}[$n$] are zero for $n\geq 4$. Here are the equations for $n∈\{1,2,3\}$:
\begin{align*}
\eqref{ainfrel}[1]: && 0 \eqs& m_1\circ m_1\\
\eqref{ainfrel}[2]: && 0 \eqs& m_1\circ m_2 - m_2\circ (m_1\otimes 1+1\otimes m_1)\\
\eqref{ainfrel}[3]: && 0 \eqs& m_1\circ m_3 + m_2\circ (1\otimes m_2-m_2\otimes1)\\
&&& + m_3\circ (m_1\otimes 1^{\otimes 2}+1\otimes m_2\otimes 1+1^{\otimes 2}\otimes m_1)
\\&& 
\ovs{m_3=0}& m_2\circ (1\otimes m_2-m_2\otimes1)
\end{align*}
So \eqref{ainfrel}[$1$] ensures that $m_1$ is a differential. Moreover, \eqref{ainfrel}[$3$] states that $m_2$ is an associative binary operation, since for homogeneous $x,y,z∈A$ we have  $0=m_2\circ \mbox{$(1\otimes m_2-m_2\otimes1)(x\otimes y\otimes z)$} = m_2(x\otimes m_2(y\otimes z) - m_2(x\otimes y)\otimes z)$, where because of $|m_2|=0$ there are no additional signs caused by the Koszul sign rule. Equation \eqref{ainfrel}[$2$] is the Leibniz rule.
\end{bsp}
\begin{bsp}[$\A_n$-morphisms induce complex morphisms]
\label[bsp]{ex:cc} $ $\\*
Let $n∈\Z_{\geq 1}\cup \{∞\}$. Let $(A', (m'_k)_{k∈[1,n]})$ and $(A, (m_k)_{k∈[1,n]})$ be two $\A_n$-algebras and let $(f_k)_{k∈[1,n]}: (A', (m'_k)_{k∈[1,n]}) → (A, (m_k)_{k∈[1,n]})$ be an $\A_n$-morphism.

By \eqref{ainfrel}[$1$], $(A',m'_1)$ and $(A,m_1)$ are complexes. Equation \eqref{finfrel}[$1$] is
\begin{align*}
f_1\circ m'_1 \eqs& m_1\circ f_1.
\end{align*}
Thus $f_1:(A',m'_1)→(A,m_1)$ is a complex morphism.

For $n\geq 2$, we have also \eqref{finfrel}[$2$]:
\begin{align}
\label{finfrel2}
f_1\circ m'_2 - f_2\circ(m'_1\otimes 1 + 1\otimes m'_1) \eqs& m_1\circ f_2 + m_2\circ (f_1\otimes f_1)
\end{align}
\end{bsp}
Recall the conventions concerning $\Hom_B^k(C,C')$.
\begin{lemma}[{cf.\ e.g.\ \cite[Section 3.3]{Ke01}}]
\label[lemma]{lem:cai}
Let  $B$ be an (ordinary)  $R$-algebra and $M=((M_i)_{i∈\Z}, (d_i)_{i∈\Z})$ a complex of $B$-modules, that is a sequence $(M_i)_{i∈\Z}$ of $B$-modules and $B$-linear maps $d_i:M_i→M_{i-1}$ such that $d_{i-1}\circ d_i = 0$ for all $i∈\Z$.
Let 
\begin{align*}
\Hom^i_B(M,M):\eqs& \prod_{z∈\Z}\Hom_B(M_{z+i},M_z)\\
\eqs& \{g=(g_z)_{z∈\Z} \mid  g_z∈\Hom_B(M_{z+i},M_z)\text{ for }z∈\Z\}.
\end{align*}
Then 
\begin{align*}
A=\Hom^*_B(M,M) := \bigoplus_{i∈\Z} \Hom^i_B(M,M)
\end{align*}
 is a graded  $R$-module. 
We have $d := (d_{z+1})_{z∈\Z}=\sum_{i∈\Z} \ls d_{i+1}\rs_{i+1}^{i} ∈\Hom^1_B(M,M)$. We define $m_1:=d_{\Hom^*(M,M)}:A→A$, that is  for homogeneous $g∈A$ we have
\begin{align*}
m_1(g) = d \circ g -(-1)^{|g|}  g\circ d.
\end{align*}
We define $m_2:A^{\otimes 2}→A$ for homogeneous $g,h∈A$ to be composition, i.e.\ 
\begin{align*}
m_2(g\otimes h) := g \circ h.
\end{align*}
For $n\geq 3$ we set $m_n:A^{\otimes n}→A$, $m_n=0$. Then $(m_n)_{n\geq 1}$ is an $\A_\infty$-algebra structure on $A=\Hom^*_B(M^*,M^*)$. More precisely, $(A, (m_n)_{n\geq 1})$ is a dg-algebra.
\end{lemma}
\kommentar{
\begin{proof}
Because of $|d|=1$ we have $|m_1| = 1 = 2-1$. The graded map $m_2$ has  degree $0 = 2-2$. The other maps $m_n$ are zero and have therefore automatically correct degree. As discussed in \cref{ex:dg} we only need to check \eqref{ainfrel}[$n$] for $n=1,2,3$. 
Equation \eqref{ainfrel}[$1$] holds  because for homogeneous $g∈A$ we have
\begin{align*}
m_1(m_1(g)) \eqs& m_1[d\circ g - (-1)^{|g|}g\circ d]\\
 \eqs& d\circ[d\circ g - (-1)^{|g|}g\circ d] - (-1)^{|g|+1}[d\circ g - ({-}1)^{|g|}g\circ d]\circ d\\
\ovs{d^2=0}& -(-1)^{|g|}d\circ g\circ d - (-1)^{|g|+1}d\circ g\circ d = 0.
\end{align*}
Concerning \eqref{ainfrel}[$2$], we have for homogeneous $g,h∈A$
\begin{align*}
(m_2\circ (m_1\otimes 1+&1\otimes m_1))(g\otimes h)= m_2(m_1(g)\otimes h + (-1)^{|g|}g\otimes m_1(h))\\
\eqs& (d\circ g - (-1)^{|g|}g\circ d)\circ h + (-1)^{|g|}g\circ(d\circ h - (-1)^{|h|}h\circ d)\\
\eqs& d\circ g \circ h - (-1)^{|g|+|h|}g\circ h\circ d\\
\eqs& (m_1\circ m_2 )(g\otimes h).
\end{align*}
The map $m_2$ is induced by the composition of morphisms which is associative. As discussed in \cref{ex:dg}, equation \eqref{ainfrel}[$3$] holds.
\end{proof}
}

\begin{bem}
\label[bem]{bem:di}
In $\Hom^*(\pres\fp,\pres\fp)$ we have (cf. \eqref{prfp})
\begin{align*}
d \eqs& \sum_{i\geq 0} \ls e_{ω(i),ω(i+1)}\rs_{i+1}^i\,.
\end{align*}
\end{bem}
\begin{Def}[Homology of $\A_∞$-algebras, quasi-isomorphisms, minimality, minimal models]
\label[Def]{def:homology}
As $m_1^2=0$ (cf.\ \eqref{ainfrel}[$1$]) and 
$|m_1|=1$, we have the complex
\[\cdots → A^{i-1}\xrightarrow{m_1|_{A^{i-1}}}A^{i}\xrightarrow{m_1|_{A^i}} A^{i+1}→\cdots\,.\]
We define $\Hm^k A:= \ker(m_1|_{A^k})/\im(m_1|_{A^{k-1}})$ and $\Hm^* A := \bigoplus_{k∈\Z} \Hm^k A$, which gives the homology of $A$ the structure of a graded $R$-module.

A morphism of $\A_∞$-algebras $(f_k)_{k\geq 1}: (A', (m'_k)_{k\geq 1}) → (A, (m_k)_{k\geq 1})$ is called a \textit{quasi-isomorphism} if the morphism of complexes $f_1:(A', m'_1)→(A,m_1)$ (cf. \cref{ex:cc}) is a quasi-isomorphism.

An $\A_∞$-algebra is called \textit{minimal}, if $m_1=0$. If $A$ is an $\A_∞$-algebra and $A'$ is a minimal $\A_∞$-algebra quasi-isomorphic to $A$, then $A'$ is called a \textit{minimal model of $A$}.
\end{Def}
The existence of minimal models is assured by the following theorem. 
\begin{tm}(minimality theorem, cf.\ \cite{Ke01ad} (history), \cite{Ka82}, \cite{Ka80}, \cite{Pr84}, \cite{GuLaSt91}, \cite{JoLa01}, \cite{Me99}, … )
\label[tm]{tm:kadeishvili}
Let $(A,(m_k)_{k\geq 1})$ be an $\A_\infty$-algebra such that the homology $\Hm^* A$ is a projective $R$-module. Then there exists an $\A_\infty$-algebra structure $(m'_k)_{k\geq 1}$ on $\Hm^*A$ and a quasi-isomorphism of $\A_∞$-algebras $(f_k)_{k\geq 1}:(\Hm^* A,(m'_k)_{k\geq 1})→(A,(m_k)_{k\geq 1})$, such that
\begin{itemize}
\item $m'_1 = 0$ and
\item the complex morphism $f_1:(\Hm^*A,m'_1)→(A, m_1)$ induces the identity in homology. I.e.\ each element $x∈\Hm^* A$, which is a homology class of $(A,m_1)$, is mapped by $f_1$ to a representing cycle.
\end{itemize}
\end{tm}

%

For constructing $\A_∞$-structures induced by another $\A_∞$-algebra, we have the following
\begin{lemma}[{cf.\  \cite[Proof of Theorem 1]{Ka80}}]
\label[lemma]{lem:aaut}
Let $n∈\Z_{\geq 1}\cup \{∞\}$. Let $(A',(m'_k)_{k∈[1,n]})$ be a pre-$\A_n$-algebra. Let $(A,(m_k)_{k∈[1,n]})$ be an $\A_n$-algebra. Let $(f_k)_{k∈[1,n]}$ be a pre-$\A_n$-morphism from $A'$ to $A$ such that \eqref{finfrel}$[k]$ holds for $k∈[1,n]$. Suppose $f_1$ to be injective.\\*
Then $(A',(m'_k)_{k∈[1,n]})$ is  an $\A_n$-algebra and $(f_k)_{k∈[1,n]}$ is a morphism of $\A_n$-algebras from $(A',(m'_k)_{k∈[1,n]})$ to $(A,(m_k)_{k∈[1,n]})$.
\end{lemma}
This results from the bar construction and a straightforward induction on $n$.

\begin{lemma}[{\cite[Theorem 5]{Ve10}}]
\label[lemma]{lem:finffinite}
Let $R$ be a commutative ring and $(A,(m_n)_{n\geq 1})$ be a dg-algebra (over $R$). Suppose given a graded $R$-module $B$ and graded maps $f_n:B^{\otimes n}→A$, $m'_n:B^{\otimes n}→B$ for $n\geq 1$. Suppose given $k\geq 1$ such that
we have $f_i=0$ for $i\geq k$, we have $m'_i=0$ for $i\geq k+1$, 
and  \eqref{finfrel}$[n]$ is satisfied for $n∈[1,2k-2]$. Then \eqref{finfrel}$[n]$ is satisfied for all $n\geq 1$.
\end{lemma} 

\subsection{The homology of \texorpdfstring{$\Hom^*_{\fp\!\Sp}(\pres\fp,\pres\fp)$}{Hom*(FpSp)(PResFp,PResFp)}}
\label{subsec:homology}
We need a well-known result of homological algebra in a particular formulation:
\begin{lemma}
\label[lemma]{prpr-prm}
Let $F$ be a field. Let  $B$ be an $F$-algebra. Let $M$ be a $B$-module. Let $Q = (\cdots \rightarrow Q_2 \xrightarrow{d_2} Q_1 \xrightarrow{d_1} Q_0 → 0→
\cdots)$ be a projective resolution of $M$ with augmentation $ε:Q_0→M$. Then we have maps for $k∈\Z$
\begin{align*}
Ψ_k:\Hom^k_B(Q,Q) &\xrightarrow{}\Hom^k_B(Q, M) := \Hom_B(Q_k,M)\\
(g_i:Q_{i+k}→Q_i)_{i∈\Z} &\mapsto ε\circ g_0
\end{align*}
The right side is equipped with the differentials (dualization of $d_k$)
\begin{align*}
(d_k)^*:\Hom_B(Q_k,M) &\rightarrow \Hom_B(Q_{k+1},M)\\
g & \mapsto (-1)^k g\circ d_k
\end{align*}
and the left side is equipped with the differential $m_1$ of its dg-algebra structure, cf.\ \cref{lem:cai}.

Then $(Ψ_k)_{k∈\Z}$ becomes a complex morphism from the complex  $\Hom^*_B(Q,Q)$ to the complex $\Hom^*_B(Q, M)$ that induces isomorphisms $\bar{Ψ}_k$ of $F$-vector spaces on the homology
\begin{align*}
\bar{Ψ}_k:\Hm^k\Hom^*_B(Q,Q) & \xrightarrow{\simeq}\Hm^k \Hom^*_B(Q, M)\\
\overline{(g_i:Q_{i+k}→Q_i)_{i∈\Z}}&\mapsto \overline{ε\circ g_0}
\end{align*}
\end{lemma}
\cref{prpr-prm} is  \cite[§5 Proposition 4a)]{Bo80} applied to the quasi-isomorphism induced by the augmentation, cf.\ \cite[§3 Définition 1]{Bo80}.


Recall the notation $\ls x\rs_y^z$ for the description of elements of $\Hom_B^k(C,C')$.
\begin{pp}
\label[pp]{pp:iota}
Recall that $p\geq 3$ is a prime and $l=2(p-1)$. \\* Write   \mbox{$A := \Hom^*_{\fp\!\Sp}(\pres\fp,\pres\fp)$}.
Let 
\begin{align*}
ι:\eqs&\dum_{i\geq 0} \ls e_{ω(i)} \rs_{i+l}^i =\dum_{i\geq 0}\dum_{k=0}^{l-1}\ls e_{ω(k)} \rs_{(i+1)l+k}^{il+k} ∈ A^l
\\
χ :\eqs& \dum_{i\geq 0}\left(\ls e_1 \rs_{il+l-1}^{il} + \left\lgroup\dum_{k=1}^{p-2}\ls e_{k+1,k}\rs_{il+l-1+k}^{il + k}\right\rgroup\right.\\
&\left. + 
\ls e_{p-1}\rs_{il+l-1 +(p-1)}^{il+(p-1)}
+ \left\lgroup\dum_{k=1}^{p-2}\ls e_{p -k-1,p-k}\rs_{il+l-1+(p-1)+k}^{il +(p-1)+ k}\right\rgroup\right)∈A^{l-1}.
\end{align*}
\begin{enumerate}[\rm (a)]
\item For $j\geq 0$, we have 
$
ι^j = \dum_{i\geq 0} \ls e_{ω(i)} \rs_{i+jl}^i= \dum_{i\geq 0}\dum_{k=0}^{l-1}\ls e_{ω(k)} \rs_{(i+j)l+k}^{il+k}\,.
$
\item Suppose given $y\geq 0$. Let $h∈A^y$ be $l$-periodic, that is 
$h=\dum_{i\geq 0} \dum_{k=0}^{l-1} \ls h_k\rs_{il+k+y}^{il + k}$.
Then for $j\geq 0$, we have 
\begin{align*}
h\circ ι^j = ι^j\circ h = \dum_{i\geq 0} \dum_{k=0}^{l-1} \ls h_k\rs_{(i+j)l + k +y}^{il+k}∈A^{y+jl}.
\end{align*}
\item Suppose given $y∈\Z$. For $h∈A^y$ and $j\geq 0$, we have $m_1(h\circ ι^j) = m_1(h)\circ ι^j$.
\item For $j\geq 0$, we have $m_1(ι^j)=0$. Thus $ι^j$ is a cycle.
\item For $j\geq 0$, we have 
\begin{align*}
χι^j :\eqs& χ\circ ι^j = ι^j\circ χ\\
\eqs&  \dum_{i\geq 0}\left(\ls e_1 \rs_{(i+j+1)l-1}^{il} + \left\lgroup\dum_{k=1}^{p-2}\ls e_{k+1,k}\rs_{(i+j+1)l-1+k}^{il + k}\right\rgroup\right.\\
&\left. + \ls e_{p-1}\rs_{(i+j+1)l-1 +(p-1)}^{il+(p-1)}
+ \left\lgroup\dum_{k=1}^{p-2}\ls e_{p -k-1,p-k}\rs_{(i+j+1)	l-1+(p-1)+k}^{il +(p-1)+ k}\right\rgroup\right)∈A^{jl+l-1}.
\end{align*}
For convenience, we also define $χ^0ι^j:= ι^j$ and $χ^1ι^j:=χι^j=χ\circ ι^j$ for $j\geq 0$.
\item For $j\geq 0$, we have $m_1(χι^j) = 0$. Thus $χι^j$ is a cycle.
\item 
Suppose given $k∈\Z$. A $\fp$-basis of $\Hm^kA$ is given by
\begin{align*}
\{\overline{ι^j}\} & \text{ if $k=jl$ for some $j\geq 0$}\\
\{\overline{χι^j}\} & \text{ if $k=jl+l-1$ for some $j\geq 0$} \\
∅ &\text{ else}.
\end{align*}
Thus the set $\mathfrak{B}:=\{\overline{ι^j}\mid j\geq 0\}\sqcup \{\overline{χι^j}\mid j\geq 0\}$ is an $\fp$-basis of $\Hm^*A = \bigoplus_{z∈\Z} \Hm^zA$.
\end{enumerate}
\end{pp}

\begin{proof}
The element $ι$ is well-defined since $ω(y) = ω(l+y)$ for $y\geq 0$.\\*
In the definition of $χ$ we need to check that the "$\ls * \rs_*^*$" are well defined. This is easily proven by calculating the $ω(y)$ where $y$ is the lower respective upper index of  "$\ls * \rs_*^*$".

(a): As $\pr_i = \{0\}$ for $i<0$, the identity element of $A$  is given by
$ι^0 = \dum_{i\geq 0}\ls e_{ω(i)}\rs_{i}^i$,
which agrees with the assertion in case $j=0$. So we have proven the induction basis for induction on $j$. So now assume that for some $j\geq 0$ the assertion holds. Then
\begin{align*}
ι^{j+1} \eqs& ι\circ ι^j = \left(\dum_{i\geq 0} \ls e_{ω(i)} \rs_{i+l}^i\right)\circ\left(\dum_{i'\geq 0} \ls e_{ω(i')} \rs_{i'+jl}^{i'}\right) \\
\eqs& \dum_{i\geq 0} \ls e_{ω(i)}\circ e_{ω(i+l)} \rs_{i+l+jl}^i = \dum_{i\geq 0} \ls e_{ω(i)}\rs_{i+(j+1)l}^i\,.
\end{align*}
Thus the proof by induction is complete.

(b): 
We have
\begin{align*}
ι^j \circ h \eqs& \left(\sum_{i\geq 0} \sum_{k=0}^{l-1}\ls e_{ω(il+k)} \rs_{(i+j)l+k}^{il+k}\right) \circ \left(\sum_{i'\geq 0}\sum_{k'=0}^{l-1}\ls h_{k'}\rs_{i'l+k'+ y}^{i'l+k'}\right)
\overset{\natop{i'\leadsto i+j}{k'\leadsto k}}{=} \sum_{i\geq 0} \sum_{k=0}^{l-1} \ls h_k\rs_{(i+j)l+k+y}^{il+k}\\
h\circ ι^j \eqs& \left(\sum_{i\geq 0}\sum_{k=0}^{l-1}\ls h_k\rs_{il+k+y}^{il+k}\right) \circ
 \left(\sum_{i'\geq 0} \ls e_{ω(i')} \rs_{i'+jl}^{i'}\right)
= \sum_{i\geq 0}\sum_{k=0}^{l-1} \ls h_k\rs_{(i+j)l + k+y}^{il+k}\, .
\end{align*}
So we have proven (b). 

(c): The differential $d$ of $\pres\fp$ is $l$-periodic (cf. \cref{bem:di}) and thus
\begin{align*}
m_1(h)\circ ι^j \eqs& (d\circ h - (-1)^y h\circ d)\circ ι^j\\
\ovs{(b), |ι^j|\equiv_2 0}& d\circ h\circ ι^j - (-1)^{y+|ι^j|} h\circ ι^j\circ d = m_1(h\circ ι^j).
\end{align*}

(d): We have
$
m_1(ι^j) \overset{(c)}{=} m_1(ι^0)\circ ι^j = (d\circ ι^0 - (-1)^0 ι^0 d)\circ ι^j = (d - d)\circ ι^j = 0.
$

(e) is implied by (b) using the fact that $χ$ is $l$-periodic.

(f): Because of (c) we have $m_1(χι^j) = m_1(χ)\circ ι^j$. Because $|χ|=l-1$ is odd we have
\begin{align*}
m_1&(χ) = d\circ χ - (-1) χ\circ d = χ\circ d + d\circ χ\\
\ovs{\text{R.}\ref{bem:di}}& \Big(\dum_{i\geq 0}\left(\ls e_1 \rs_{il+l-1}^{il} + \left\lgroup\dum_{k=1}^{p-2}\ls e_{k+1,k}\rs_{il+l-1+k}^{il + k}\right\rgroup+ \ls e_{p-1}\rs_{il+l-1 +(p-1)}^{il+(p-1)}\right.\\
&\left. 
+ \left\lgroup\dum_{k=1}^{p-2}\ls e_{p -k-1,p-k}\rs_{il+l-1+(p-1)+k}^{il +(p-1)+ k}\right\rgroup\right)\Big)
\circ \left(\dum_{y\geq 0} \ls e_{ω(y),ω(y+1)}\rs_{y+1}^y\right)\\
&+\left(\dum_{y\geq 0} \ls e_{ω(y),ω(y+1)}\rs_{y+1}^y\right) \circ \Big(
\dum_{i\geq 0}\left(\ls e_1 \rs_{il+l-1}^{il} + \left\lgroup\dum_{k=1}^{p-2}\ls e_{k+1,k}\rs_{il+l-1+k}^{il + k}\right\rgroup\right.\\
&\left. + \ls e_{p-1}\rs_{il+l-1 +(p-1)}^{il+(p-1)}
+ \left\lgroup\dum_{k=1}^{p-2}\ls e_{p -k-1,p-k}\rs_{il+l-1+(p-1)+k}^{il +(p-1)+ k}\right\rgroup\right)\Big)\\
\eqs& 
\dum_{i\geq 0}\Big(
\ls e_1\circ e_{1,1}\rs_{il+l}^{il} + \left\lgroup\dum_{k=1}^{p-2}\ls e_{k+1,k}\circ e_{k,k+1}\rs_{il+l+k}^{il+k}\right\rgroup\\
&+\ls e_{p-1}\circ e_{p-1,p-1}\rs_{il+l+(p-1)}^{il+(p-1)}
+ \left\lgroup\dum_{k=1}^{p-2}\ls e_{p-k-1,p-k}\circ e_{p-k,p-k-1}\rs_{il+l+(p-1)+k}^{il+(p-1)+k}\right\rgroup\Big)\\
&+\dum_{i\geq 1} \ls e_{1,1}\circ e_1\rs_{il+l-1}^{il-1}
+ \dum_{i\geq 0}\Big( \left\lgroup\dum_{k=1}^{p-2}\ls e_{k,k+1}\circ e_{k+1,k}\rs_{il+l+k-1}^{il+k-1}\right\rgroup\\
&+ \ls e_{p-1,p-1}\circ e_{p-1}\rs_{il+l-1+(p-1)}^{il-1+(p-1)}
+ \left\lgroup\dum_{k=1}^{p-2}\ls e_{p-k,p-k-1}\circ e_{p-k-1,p-k}\rs_{il+l-1+(p-1)+k}^{il-1+(p-1) + k}\right\rgroup\Big)\\
\overset{*}{\eqs}& \dum_{i\geq 0}\Big(
\ls e_{1,1} + e_{1,2}\circ e_{2,1}\rs_{il+l}^{il} + \left\lgroup\dum_{k=1}^{p-3}\ls e_{k+1,k}\circ e_{k,k+1}+ e_{k+1,k+2}\circ e_{k+2,k+1}\rs_{il+l+k}^{il+k}\right\rgroup\\
&+ \ls e_{p-1,p-2}\circ e_{p-2,p-1} + e_{p-1,p-1}\rs_{il+l+p-2}^{il+p-2}
+ \ls e_{p-1,p-1}+e_{p-1,p-2}\circ e_{p-2,p-1}\rs_{il+l+p-1}^{il+p-1}\\
&+ \left\lgroup\dum_{k=1}^{p-3}\ls e_{p-k-1,p-k}\circ e_{p-k,p-k-1} + e_{p-k-1,p-k-2}\circ e_{p-k-2,p-k-1}\rs_{il+l+p-1+k}^{il+p-1+k}\right\rgroup\\
&+ \ls e_{1,2}\circ e_{2,1} + e_{1,1}\rs_{(i+1)l+l-1}^{(i+1)l-1}\Big)
\overset{\text{L.}\ref{lem:relfp}(a)}{=} 0
\end{align*}
In the step marked by "$*$" we sort the summands by their targets. Note that when splitting sums of the form $\dum_{k=1}^{p-2}(…)_k$ into $(…)_1+\dum_{k=2}^{p-2}(…)_k$ or into $(…)_{p-2}+\dum_{k=1}^{p-3}(…)_k$, the existence of the summand that is split off is ensured by $p\geq 3$.

(g): We first show that the differentials of the complex $\Hom^*(\pres \fp, \fp)$ (cf. \cref{prpr-prm}) are all zero:
By \cref{lem:relfp}, $\{ε\}$ is an $\fp$-basis of $\Hom_{\fp\!\Sp}(P_1,\fp)$, and for $k∈[2,p-1]$ we have $\Hom_{\fp\!\Sp}(P_k,\fp) = 0$. So the only non-trivial $(d_k)^*$ are those where $\pr_k = \pr_{k+1} = P_1$. This is the case only when $k=lj + l-1$ for some $j\geq 0$. Then $d_k= e_{1,1}$. For $ε∈\Hom(P_1,\fp)$, we have $(d_k)^* (ε) = (-1)^kε\circ e_{1,1} \overset{\text{L.}\ref{lem:relfp}(a)}{=} 0$. As $ \Hom(P_1,\fp)= \langle ε\rangle_\fp$, we have $(d_k)^* = 0$.

So $\Hm^k \Hom^*(\pres \fp, \fp) = \Hom^k(\pres \fp, \fp)$. We use \cref{prpr-prm}.\\*
For $k=jl$, $j\geq 0$, we have $\bar{Ψ}^k(\overline{ι^j}) \overset{(a)}{=} ε$, and $\{ε\}$ is a basis of $\Hm^k \Hom^*(\pres \fp, \fp)$.\\*
For $k=jl+l-1$, $j\geq 0$, we have $\bar{Ψ}^k(\overline{χι^j})\overset{(e)}{=}ε$, and $\{ε\}$ is a basis of $\Hm^k \Hom^*(\pres \fp, \fp)$.\\*
Finally, for $k=jl+r$ for some $j\geq 0$ and some $r∈[1,l-2]$ and for $k<0$, we have $\Hm^k \Hom^*(\pres \fp, \fp)=\{0\}$.
\end{proof}
\subsection[An \texorpdfstring{$\A_\infty$}{A(oo)}-structure on \texorpdfstring{$\Ext^*_{\fp\!\Sp}\!(\fp,\fp)$}{Ext*(FpSp)(Fp,Fp)} as a minimal model of \texorpdfstring{$\Hom^*_{\fp\!\Sp}$}{Hom*(FpSp)}\discretionary{}{}{}\texorpdfstring{$(\pres\fp,$}{(PResFp,)}\discretionary{}{}{}\texorpdfstring{$\pres\fp)$}{PresFp}]
{An \texorpdfstring{$\A_\infty$}{A(oo)}-structure on \texorpdfstring{$\Ext^*_{\fp\!\Sp}(\fp,\fp)$}{Ext*(FpSp)(Fp,Fp)} as a minimal model of \texorpdfstring{$\Hom^*_{\fp\!\Sp}(\pres\fp,\pres\fp)$}{Hom*(FpSp)(PResFp,PresFp)}}
\label{subsec:minmod}
Recall that $p\geq 3$ is a prime. Write $A := \Hom^*_{\fp\!\Sp}(\pres\fp,\pres\fp)$, which becomes an $\A_∞$-algebra $(A, (m_n)_{n\geq 1})$ over $R=\fp$ via \cref{lem:cai}. We implement $\Ext^*_{\fp\!\Sp}(\fp,\fp)$ as $\Ext^*_{\fp\!\Sp}(\fp,\fp):=\Hm^*A$.

Our goal in this section is to construct an $\A_\infty$-structure $(m'_n)_{n\geq 1}$ on $\Hm^*A$ and a morphism of $\A_\infty$-algebras $f=(f_n)_{n\geq 1}:(\Hm^*A,(m'_n)_{n\geq 1})→(A, (m_n)_{n\geq 1})$ which satisfy the statements of \cref{tm:kadeishvili}. I.e.\ we will construct a minimal model of $A$.
%
In preparation of the definitions of the $f_n$ and $m'_n$, we name and examine certain elements of $A$:
\begin{lemma}
\label[lemma]{lem:gamma}
Suppose given $k∈[2,p-1]$. We set 
\begin{align*}
γ_k:= \dum_{i\geq 0}\left( \ls e_k \rs_{k(l-1)+li}^{k-1+li}  
+\ls e_{p-k} \rs_{k(l-1) + (p-1) + li}^{k-1+(p-1)+li}\right)∈A^{k(l-2)+1}.
\end{align*}
For $j\geq 0$, we have
\begin{align*}
γ_kι^j := γ_k\circ ι^j = ι^j\circ γ_k = \dum_{i\geq 0}\left( \ls e_k \rs_{k(l-1)+l(i+j)}^{k-1+li}
+ \ls e_{p-k} \rs_{k(l-1) + (p-1) + l(i+j)}^{k-1+(p-1)+li}\right) ∈A^{k(l-2)+1+jl}
\end{align*}
and
\begin{align*}
m_1(γ_kι^j)\eqs&\dum_{i\geq 0}\left( \ls e_{k-1,k} \rs_{k(l-1)+l(i+j)}^{k-2+li} 
+ 
 \ls e_{p-k+1,p-k}\rs_{k(l-1) + (p-1) + l(i+j)}^{k-2+(p-1)+li}\right.\\
&\left.+
 \ls e_{k,k-1}\rs_{k(l-1)+1+l(i+j)}^{k-1+li}
+ 
\ls e_{p-k,p-(k-1)}\rs_{k(l-1) + p + l(i+j)}^{k-1+(p-1)+li}\right).
\end{align*}
\end{lemma}
\begin{proof}
We need to prove that $γ_k$ is well-defined. Let $i\geq 0$.\\*
We consider the first term. The complex $\pres \fp$ (cf.\ \eqref{prfp}, \eqref{omega}) has entry $P_k$ at position $k(l-1)+li$ and at position $k-1+li$: 
We have $k(l-1)+li = (k-1+i)l + l - k$. So $ω(k(l-1)+li) = l - (l - k) = k$ since  $p-1\leq l-k \leq l-1$. 
We have $ω(k-1+li) = (k-1)+1 = k$ since $0\leq k-1 \leq p-2$. As $k(l-1)+li, k-1+li\geq 0$, we have $\pr_{k(l-1)+li}=P_{ω(k(l-1)+li)}=P_k$ and $\pr_{k-1+li} = P_{ω(k-1+li)} = P_k$. So the first term is well-defined.\\*
Now consider the second term. The complex $\pres \fp$ has entry $P_{p-k}$ at position $k(l-1) + (p-1) + li$ and at position $k-1+(p-1)+li$: 
We have $k(l-1) + (p-1) + li = (i+k)l + (p-1)-k$, so $ω(k(l-1) + (p-1)+li) = (p-1)-k + 1 = p-k$ since  $0\leq (p-1)-k \leq p-2$. We have $ω(k-1+(p-1)+li) = 2(p-1) - (k-1) -(p-1) = p-k$ since $p-1\leq k-1 + (p-1) \leq 2(p-1)-1$. As $k(l-1) + (p-1) + li, k-1+(p-1)+li\geq 0$, we have  $\pr_{k(l-1)+(p-1)+li} = P_{ω(k(l-1)+(p-1)+li)} = P_{p-k}$ and $\pr_{k-1+(p-1)+li} = P_{ω(k-1+(p-1)+li)} = P_{p-k}$. So the second term is well-defined.

The degree of the tuple of  maps is computed to be $(k(l-1) + li) - (k-1+li) = k(l-2) + 1 = (k(l-1)+(p-1)+li) - (k-1+(p-1)+li)$.

The explicit formula for $γ_kι^j$ is an application of \cref{pp:iota}(b).

The degree $|γ_kι^j| = k(l-2) + 1$ is odd, so
\begin{align*}
m_1(γ_kι^j)\,\, \ovs{\text{L.}\ref{lem:cai}}& d \circ γ_kι^j + γ_kι^j \circ d\\
\ovs{\text{R.}\ref{bem:di}}& \dum_{i\geq 0} \ls e_{ω(k-2),ω(k-1)}\rs_{k-1+li}^{k-2+li} \circ \dum_{i\geq 0} \ls e_k \rs_{k(l-1)+l(i+j)}^{k-1+li} \\
&+ \dum_{i\geq 0} \ls e_{ω(p-1+k-2),ω(p-1+k-1)}\rs_{k-1+(p-1)+li}^{k-2+(p-1)+li}\circ \dum_{i\geq 0} \ls e_{p-k} \rs_{k(l-1) + (p-1) + l(i+j)}^{k-1+(p-1)+li} \\
&+\dum_{i\geq 0} \ls e_k \rs_{k(l-1)+l(i+j)}^{k-1+li} \circ \dum_{i\geq 0} \ls e_{ω(l-k),ω(l-k+1)}\rs_{k(l-1)+1+l(i+j)}^{k(l-1)+l(i+j)}\\
&+ \dum_{i\geq 0} \ls e_{p-k} \rs_{k(l-1) + (p-1) + l(i+j)}^{k-1+(p-1)+li}\circ
 \dum_{i\geq 0} \ls e_{ω(p-1-k),ω(p-k)}\rs_{k(l-1) + p + l(i+j)}^{k(l-1) + (p-1) + l(i+j)}\\
\eqs&\dum_{i\geq 0} \ls e_{k-1,k} \rs_{k(l-1)+l(i+j)}^{k-2+li} 
+ \dum_{i\geq 0} \ls e_{p-k+1,p-k}\rs_{k(l-1) + (p-1) + l(i+j)}^{k-2+(p-1)+li}\\
&+\dum_{i\geq 0} \ls e_{k,k-1}\rs_{k(l-1)+1+l(i+j)}^{k-1+li}
+ \dum_{i\geq 0} \ls e_{p-k,p-(k-1)}\rs_{k(l-1) + p + l(i+j)}^{k-1+(p-1)+li}\\
\end{align*} 
Note that in the second line  $k-2 +li\geq 0$ as $i\geq 0$ and $k\geq 2$.
\end{proof}
\begin{lemma}
\label[lemma]{lem:chichi}
For $j,j'\geq 0$, we have
$
χι^j \circ χι^{j'} = m_1(γ_2ι^{j+j'}).
$
\end{lemma}
\begin{proof}
It suffices to prove that $χ\circ χ = m_1(γ_2)$ since then $χι^j \circ χι^{j'} \overset{\text{P.}\ref{pp:iota}(e)}{=} χ\circ χ\circ ι^{j+j'} = m_1(γ_2)\circ ι^{j+j'} \overset{\text{P.}\ref{pp:iota}(c)}{=} m_1(γ_2ι^{j+j'})$.\\*
To determine when a composite is zero, we will need the following.  For \mbox{$0\leq k,k' < l$}, we examine the condition 
\begin{align}
\label{mulccg}
il +l - 1 + k = i'l + k'.
\end{align}
If $k=0$ then \eqref{mulccg} holds iff $i=i'$ and $k' = l-1$.\\*
If $k\geq 1$ then \eqref{mulccg} holds iff $i+1=i'$ and $k'=k-1$.\\*
So
\begin{align*}
χ\circ χ \ovs{p\geq 3\vphantom{I_{I_y}}}&
\left(\dum_{i\geq 0}\left(\ls e_1 \rs_{il+l-1}^{il} + \ls e_{2,1}\rs_{il+l}^{il+1}+\left\lgroup\dum_{k=2}^{p-2}\ls e_{k+1,k}\rs_{il+l-1+k}^{il + k}\right\rgroup\right.\right.\\
& + \ls e_{p-1}\rs_{il+l-1 +(p-1)}^{il+(p-1)}
+ \ls e_{p-2,p-1}\rs_{il+l+p-1}^{il+p}
+\left.\left.\left\lgroup\dum_{k=2}^{p-2}\ls e_{p -k-1,p-k}\rs_{il+l-1+(p-1)+k}^{il +(p-1)+ k}\right\rgroup\right)\right)\\
&\circ\left(\dum_{i'\geq 0}\left(\ls e_1 \rs_{i'l+l-1}^{i'l} + \left\lgroup\dum_{k'=1}^{p-3}\ls e_{k'+1,k'}\rs_{i'l+l-1+k'}^{i'l + k'}\right\rgroup\right.\right.
 + \ls e_{p-1,p-2}\rs_{i'l+l+p-3}^{i'l+p-2}\\
& + \ls e_{p-1}\rs_{i'l+l-1 +(p-1)}^{i'l+(p-1)}
\left.\left.+ \left\lgroup\dum_{k'=1}^{p-3}\ls e_{p -k'-1,p-k'}\rs_{i'l+l-1+(p-1)+k'}^{i'l +(p-1)+ k'}\right\rgroup+ \ls e_{1,2}\rs_{i'l+l+2(p-2)}^{i'l+l-1}\right)\right)\\
\eqs& 
\dum_{i\geq 0}\Big(\ls e_1\circ e_{1,2}  \rs_{il+l+2(p-2)}^{il} + \ls e_{2,1}\circ e_1\rs_{il+2l-1}^{il+1}\\
& +\left\lgroup\dum_{k=2}^{p-2}\right.\!\underbrace{\ls e_{k{+}1,k}\circ e_{k,k{-}1}\rs_{il+2l-1+k-1}^{il + k}}_{=0\text{ by L.\ref{lem:relfp}(c)}}\left.\vphantom{\dum_k^p}\!\!\right\rgroup + \ls e_{p-1}\circ e_{p-1,p-2}\rs_{il+2l+p-3}^{il+(p-1)}\\
&+ \ls e_{p-2,p-1}\circ e_{p-1}\rs_{il+2l+p-2}^{il+p}
+\left\lgroup\dum_{k=2}^{p-2}\right.\underbrace{\ls e_{p -k-1,p-k}\circ e_{p-k,p-k+1}\rs_{il+2l-1+p-1+k-1}^{il +(p-1)+ k}}_{=0\text{ by L.\ref{lem:relfp}(c)}}\left.\vphantom{\dum_k^p}\!\right\rgroup\Big)\\
\eqs& \dum_{i\geq 0}\Big(\ls e_{1,2}  \rs_{(i+2)l-2}^{il} + \ls e_{2,1}\rs_{(i+2)l-1}^{il+1}
+\ls e_{p-1,p-2}\rs_{(i+2)l+p-3}^{il+p-1} + \ls e_{p-2,p-1}\rs_{(i+2)l+p-2}^{il+p} 
\Big)\\
\ovs{\text{L.}\ref{lem:gamma}}& m_1(γ_2)
\end{align*}
\end{proof}
Below are the definitions which will give a minimal $\A_∞$-algebra structure on $\Hm^*A$ and a quasi-isomorphism of $\A_∞$-algebras $\Hm^*A → A$. 
\begin{Def}
\label[Def]{defall}
Recall from \cref{pp:iota} that $\mathfrak{B}= \mbox{$\{\overline{ι^j}\mid j\geq 0\}\sqcup \{\overline{χι^j}\mid j\geq 0\}$} = \{\overline{χ^aι^j} \mid j\geq 0, a∈\{0,1\}\}$ is a basis of $\Hm^*A$. For $n∈\Z_{\geq 1}$, we set \begin{align*}
\mathfrak{B}^{\otimes n} :\eqs& \{\overline{χ^{a_1}ι^{j_1}}\otimes … \otimes \overline{χ^{a_n}ι^{j_n}}  ∈ (\Hm^*A)^{\otimes n} \mid  a_i∈\{0,1\} \text{ and } j_i∈\Z_{\geq 0} \text{ for all }i∈[1,n] \},
\end{align*}
 which is a basis of $(\Hm^* A)^{\otimes n}$ consisting of homogeneous elements.

For $n\geq 1$, we define the $\fp$-linear map $f_n:(\Hm^*A)^{\otimes n}→ A$ as follows:
\begin{description}
\item[Case $n=1$:] $f_1$ is given on $\mathfrak{B}$ by $f_1(\overline{ι^j}):=ι^j$ and $f_1(\overline{χι^j}):= χι^j$.
\item[Case {$n∈[2,p-1]$}:] $f_n$ is given on elements of $\mathfrak{B}^{\otimes n}$ 
 by
\begin{align*}
f_n(\overline{χ^{a_1}ι^{j_1}}\otimes … \otimes \overline{χ^{a_n}ι^{j_n}}) := \begin{cases}
0 & \text{if }∃i∈[1,n]: a_i = 0 \\
(-1)^{n-1}γ_nι^{j_1+…+j_n} & \text{if } 1=a_1 = a_2 = … = a_n
\end{cases}
\end{align*}
\item[Case {$n\geq p$}:] We set $f_n:= 0$.
\end{description}
For $n\geq 1$, we define the $\fp$-linear map $m'_n:(\Hm^*A)^{\otimes n}→ \Hm^*A$ by defining it on elements  $\overline{χ^{a_1}ι^{j_1}}\otimes … \otimes \overline{χ^{a_n}ι^{j_n}}∈\mathfrak{B}^{\otimes n}$:
\begin{description}
\item[{Case $∃i∈[1,n]: a_i = 0$:}] $ $\\*
 $m'_n(\overline{χ^{a_1}ι^{j_1}}\otimes … \otimes \overline{χ^{a_n}ι^{j_n}}) := 0$ for $n\neq 2$ and\\*
$
m'_2(\overline{χ^{a_1}ι^{j_1}} \otimes \overline{χ^{a_2}ι^{j_2}}) := \overline{χ^{a_1+a_2}ι^{j_1+j_2}}
$ (Note that $a_1+a_2∈\{0,1\}$).
\item[Case $a_1 = a_2 = … = a_n=1$:] $ $\\*
 $m'_n(\overline{χι^{j_1}}\otimes … \otimes \overline{χι^{j_n}}):= 0$ for $n\neq p$ and\\*
$m'_p(\overline{χι^{j_1}}\otimes … \otimes \overline{χι^{j_p}}) := (-1)^p\overline{ι^{p-1+j_1+…+j_p}}=-\overline{ι^{p-1+j_1+…+j_p}}$.
\end{description}
\end{Def}
Note that since $p\geq 3$, we have $m'_2(\overline{χι^{j_1}} \otimes \overline{χι^{j_2}}) = 0$ for $j_1,j_2\geq 0$.

\begin{tm}
\label[tm]{tm:statement}
The pair $(\Hm^*A, (m'_n)_{n\geq 1})$ is a minimal $\A_∞$-algebra. 
The tuple  $(f_n)_{n\geq 1}$ is an quasi-isomorphism of $\A_∞$-algebras from $(\Hm^*A, (m'_n)_{n\geq 1})$ to $(A, (m_n)_{n\geq 1})$. More precisely, \mbox{$f_1:(\Hm^* A, m'_1)→(A, m_1)$} induces the identity in homology.
\end{tm}
The proof of \cref{tm:statement} will take the remainder of \cref{subsec:minmod}. We will use \cref{lem:aaut}.
\begin{lemma}
\label[lemma]{lem:degree}
The maps $f_n$ and $m'_n$ have degree $|f_n| = 1-n$ and $|m'_n| = 2-n$.
I.e.\ $(f_n)_{n\geq 1}$ is a pre-$A_∞$-morphism from $\Hm^* A$ to $A$, and $(\Hm^*A, (m'_n)_{n\geq 1})$ is a pre-$\A_∞$-algebra.
\end{lemma}
\begin{proof}
We have $|f_1|=0$ as $|\overline{ι^j}| = |ι^j|$ and $|\overline{χι^j}|=|χι^j|$. For $n\geq p$ the map $f_n$ is  of degree $1-n$ as $f_n=0$.  For $n∈[2,p-1]$  the statement $|f_n|=1-n$ is proven by checking the degrees for the elements of the basis $\mathfrak{B}^{\otimes n}$ whose image under $f_n$ is non-zero:
\begin{align*}
|f_n(\overline{χι^{j_1}}\otimes … \otimes \overline{χι^{j_n}})| \eqs& 
|(-1)^{n-1}γ_nι^{j_1+…+j_n}| 
\overset{\text{L.}\ref{lem:gamma}}{=} (j_1+…+j_n)l + n(l-1) + 1-n\\
\eqs& 1-n+\dum_{x=1}^n |\overline{χι^{j_x}}| = 1-n + |\overline{χι^{j_1}}\otimes … \otimes \overline{χι^{j_n}}|
\end{align*}
Thus $|f_n|=1-n$ for all $n$ and we have proven the first statement.

Now we show $|m'_n|=2-n$. As before, we only need check the degrees for basis elements whose image is non-zero: For $\overline{χ^{a_1}ι^{j_1}}\otimes\overline{χ^{a_2}ι^{j_2}}$, $j_1,j_2\geq 0$, $a_1,a_2∈\{0,1\}$, $0∈\{a_1,a_2\}$, we have
\begin{align*}
|m'_2(\overline{χ^{a_1}ι^{j_1}}\otimes\overline{χ^{a_2}ι^{j_2}})| \eqs& |\overline{χ^{a_1+a_2}ι^{j_1+j_2}}| = (a_1+a_2)(l-1)+l(j_1+j_2)\\
\eqs& a_1(l-1) + j_1l + a_2(l-1) + j_2l = |\overline{χ^{a_1}ι^{j_1}}\otimes\overline{χ^{a_2}ι^{j_2}}| + (2-2).
\end{align*}
For $\overline{χι^{j_1}}\otimes \cdots \otimes \overline{χι^{j_p}}$, $j_x\geq 0$ for $x∈[1,p]$, we have
\begin{align*}
|m'_p(\overline{χι^{j_1}}\otimes \cdots \otimes \overline{χι^{j_p}})| \eqs& |\overline{ι^{p-1+j_1+…+j_p}}| = l(p-1+j_1+…+j_p)\\
\eqs& lp - l + l(j_1+…+j_p) = lp - 2p+2 + l(j_1+…+j_p)\\
 \eqs& p(l-1) + l(j_1+…+j_p) + 2-p 
= |\overline{χι^{j_1}}\otimes \cdots \otimes \overline{χι^{j_p}}| + 2-p 
\end{align*}
\end{proof}

\begin{lemma}
\label[lemma]{lem:f1}
We have $m'_1=0$. The equation \eqref{finfrel}$[1]$ holds. The complex morphism  {$f_1:(A',m'_1)→(A,m_1)$} is a quasi-isomorphism inducing the identity in homology.
\end{lemma}
\begin{proof}
The equality $m'_1=0$ follows immediately from the definition. Thus $m_1\circ f_1 = 0 = f_1\circ m'_1$. Moreover $f_1$ is a quasi-isomorphism inducing the identity in homology by construction, cf.\ \cref{pp:iota}(g).
\end{proof}

\begin{lemma}
\label[lemma]{lem:f1inj}
The map $f_1$ is injective.
\end{lemma}
\begin{proof}
The set $X:=\{ χ^aι^j \mid a∈\{0,1\}, j∈\Z_{\geq 1}\}\subseteq A$ is linearly independent, since it consists of nonzero elements of different summands of the direct sum $A=\bigoplus_{k∈\Z}\Hom^k(\pres\fp,\pres\fp)$.
The set $\mathfrak{B}$, which is a basis of $\Hm^* A$, is mapped bijectively to $X$ by $f_1$, so $f_1$ is injective.
\end{proof}

\begin{lemma}
\label[lemma]{lem:f2}
The equation \eqref{finfrel}$[2]$ holds.
\end{lemma}
\begin{proof}
As $m'_1=0$, equation \eqref{finfrel}[$2$] is equivalent to (cf.\ \eqref{finfrel2})
\begin{align*}
f_1\circ m'_2 \eqs& m_1\circ f_2 + m_2\circ (f_1\otimes f_1).
\end{align*}  
We check this equation on $\mathfrak{B}^{\otimes 2}$: Recall \cref{pp:iota,defall}.
\begin{align*}
f_1m'_2(\overline{ι^j}\otimes \overline{ι^{j'}}) \eqs& ι^{j+j'}
= m_2(f_1\otimes f_1)(\overline{ι^j}\otimes \overline{ι^{j'}})
= (m_1\circ f_2 + m_2\circ (f_1\otimes f_1))(\overline{ι^j}\otimes \overline{ι^{j'}})\\
f_1m'_2(\overline{ι^j}\otimes \overline{χι^{j'}})\eqs& χι^{j+j'}
= m_2(f_1\otimes f_1)(\overline{ι^j}\otimes \overline{χι^{j'}})\\
\eqs& (m_1\circ f_2 + m_2\circ (f_1\otimes f_1)) (\overline{ι^j}\otimes \overline{χι^{j'}})\\
f_1m'_2(\overline{χι^j} \otimes \overline{ι^{j'}}) \eqs& χι^{j+j'}
= m_2(f_1\otimes f_1)(\overline{χι^j} \otimes \overline{ι^{j'}})\\
\eqs& (m_1\circ f_2 + m_2\circ (f_1\otimes f_1))(\overline{χι^j} \otimes \overline{ι^{j'}})\\
f_1m'_2(\overline{χι^j}\otimes \overline{χι^{j'}}) \eqs &0 \overset{\text{L.}\ref{lem:chichi}}{=}
-m_1(γ_2ι^{j+j'}) + m_2(f_1\otimes f_1)(\overline{χι^j}\otimes \overline{χι^{j'}})\\
\eqs& (m_1\circ f_2 + m_2\circ (f_1\otimes f_1))(\overline{χι^j}\otimes \overline{χι^{j'}})
\end{align*}
Note that there are no additional signs due to the Koszul sign rule since $|f_1|=0$.
\end{proof}
The following results directly from \cref{defall}.
\begin{kor}
\label[kor]{kor:zero}
For $n\geq 2$ and  $a_1,…,a_n∈\{0,1\}$, $j_1,…,j_n\geq 0$, we have
\begin{align*}
f_n(\overline{χ^{a_1}ι^{j_1}}\otimes … \otimes \overline{χ^{a_n}ι^{j_n}})= f_n(\overline{χ^{a_1}}\otimes … \otimes \overline{χ^{a_n}}) \circ ι^{j_1+…+j_n}.
\end{align*} 
If there is additionally an $x∈[1,n]$ with $a_x=0$ then 
\[f_n(\overline{χ^{a_1}ι^{j_1}}\otimes … \otimes \overline{χ^{a_n}ι^{j_n}})=0.\]
\end{kor}
Equation \eqref{finfrel}[$n$] can be reformulated as
\begin{align*}
&f_1\circ m'_n+\underbrace{\sum_{\substack{n=r+s+t \\ r,t\geq 0,s\geq 1\\ s\leq n-1}} (-1)^{rs+t} f_{r+1+t}\circ (1^{\otimes r}\otimes m'_s\otimes 1^{\otimes t})}_{=:Φ_n}\\
 &= m_1\circ f_n +\underbrace{\sum_{\substack{2\leq r\leq n \\ i_1+…+i_r=n\\ i_s\geq 1}} (-1)^v m_r\circ (f_{i_1}\otimes f_{i_2}\otimes … \otimes f_{i_r})}_{=:Ξ_n}\,,
\end{align*}
where $v= \sum_{1\leq t < s\leq r}(1-i_s)i_t$.

A term of the form $f_{r+1+t}\circ (1^{\otimes r}\otimes m'_s\otimes 1^{\otimes t})$, $s\geq 3$, $r+t\geq 1$, is zero because of \cref{kor:zero} and the definition of $m'_p$. Also recall $m'_1=0$. Thus
\begin{align}
\label{eqphi}
Φ_n 
\eqs&\sum_{\mathclap{\substack{n=r+2+t\\ r,t\geq 0}}} (-1)^{2r+t}f_{n-1}\circ (1^{\otimes r}\otimes m'_2\otimes 1^{\otimes t}) = \sum_{r=0}^{n-2} (-1)^{n-r} f_{n-1}\circ (1^{\otimes r}\otimes m'_2\otimes 1^{\otimes n-r-2}).
\end{align}
Because of $m_k=0$ for $k\geq 3$, we have
\begin{align}
\label{eqxi}
Ξ_n 
\eqs& \sum_{\substack{i_1+i_2 = n \\ i_1,i_2\geq 1}}(-1)^{(1-i_2)i_1} m_2\circ (f_{i_1}\otimes f_{i_2}) = \sum_{i=1}^{n-1}(-1)^{ni}m_2\circ (f_i\otimes f_{n-i}).
\end{align}
We have proven:
\begin{lemma}
\label[lemma]{lem:finfrelmod}
For $n\geq 1$, condition \eqref{finfrel}$[n]$ is equivalent to $f_1\circ m'_n + Φ_n = m_1\circ f_n + Ξ_n$ where $Φ_n$ and $Ξ_n$ are as in \eqref{eqphi} and \eqref{eqxi}.
\end{lemma}

\begin{lemma}
\label[lemma]{lem:f3}
Condition \eqref{finfrel}$[n]$ holds for $n\geq 3$ and arguments $\overline{χ^{a_1}ι^{j_1}}\otimes…\otimes\overline{χ^{a_n}ι^{j_n}}∈\mathfrak{B}^{\otimes n}=\{\overline{χ^{a_1}ι^{j_1}}\otimes … \otimes \overline{χ^{a_n}ι^{j_n}}  ∈ (\Hm^*A)^{\otimes n} \mid  a_i∈\{0,1\} \text{ and } j_i∈\Z_{\geq 0} \text{ for all }i∈[1,n] \}$ where $0∈\{a_1,…,a_n\}$. 
\end{lemma}
\begin{proof}
%
Because of \cref{lem:finfrelmod,defall} it is sufficient to show
that 
\[Φ_n(\overline{χ^{a_1}ι^{j_1}}\otimes…\otimes\overline{χ^{a_n}ι^{j_n}}) = Ξ_n(\overline{χ^{a_1}ι^{j_1}}\otimes…\otimes\overline{χ^{a_n}ι^{j_n}})\]
if at least one $a_x$ equals $0$.
\begin{description}
\item[Case 1] At least two $a_x$ equal $0$:\\*
To show $Φ_n(\overline{χ^{a_1}ι^{j_1}}\otimes…\otimes\overline{χ^{a_n}ι^{j_n}})=0$, we show \\*
$f_{n-1}(1^{\otimes r}\otimes m'_2\otimes 1^{\otimes n-r-2})(\overline{χ^{a_1}ι^{j_1}}\otimes…\otimes\overline{χ^{a_n}ι^{j_n}})=0$ for $r∈[0,n-2]$: In case both components of the argument of $m'_2$ are of the form $\overline{χ^0ι^j}$, the result of $m'_2$ is of the form $\overline{ι^{j'}}$ (see \cref{defall}). Since $2\leq n-1$, \cref{kor:zero} implies the result of $f_{n-1}$ is zero. Otherwise at least one of the components of the argument of $f_{n-1}$ must be of the form $\overline{ι^j}$ and the result of $f_{n-1}$ is zero as well. So $Φ_n(\overline{χ^{a_1}ι^{j_1}}\otimes…\otimes\overline{χ^{a_n}ι^{j_n}})=0$.\\*
To show $Ξ_n(\overline{χ^{a_1}ι^{j_1}}\otimes…\otimes\overline{χ^{a_n}ι^{j_n}})=0$, we show $m_2(f_i\otimes f_{n-i})(\overline{χ^{a_1}ι^{j_1}}\otimes…\otimes\overline{χ^{a_n}ι^{j_n}})=0$ for $i∈[1,n-1]$: 
\begin{itemize}
\item Suppose given $i∈[2,n-2]$: The statements $a_1=…=a_i=1$ and $a_{i+1}=…=a_n=1$ cannot be true at the same time, so $f_i(…)=0$ or $f_{n-i}(…)=0$ and we have $m_2(f_i\otimes f_{n-i})(\overline{χ^{a_1}ι^{j_1}}\otimes…\otimes\overline{χ^{a_n}ι^{j_n}})=0$.
\item Suppose that $i=1$. Because at least two $a_x$ equal $0$ the statement $a_2=…=a_n=1$ cannot be true. Since $n-1\geq 2$, we have $f_{n-1}(…)=0$ and $m_2(f_1\otimes f_{n-1})(\overline{χ^{a_1}ι^{j_1}}\otimes…\otimes\overline{χ^{a_n}ι^{j_n}})=0$.
\item The case $i=n-1$ is analogous to the case $i=1$.
\end{itemize}
So we have $Φ_n(\overline{χ^{a_1}ι^{j_1}}\otimes…\otimes\overline{χ^{a_n}ι^{j_n}}) = 0=Ξ_n(\overline{χ^{a_1}ι^{j_1}}\otimes…\otimes\overline{χ^{a_n}ι^{j_n}})$.
\item[Case 2a] Exactly one $a_x$ equals $0$, where $x∈[2,n-1]$.\\*
We have $Φ_n(\overline{χ^{a_1}ι^{j_1}}\otimes…\otimes\overline{χ^{a_n}ι^{j_n}})=0$: In case $n\geq p+1$, it follows from $f_{n-1}=0$. Let us check the case $n∈[3,p]$: Because of \cref{defall}, we have \mbox{$f_{n-1}(1^{\otimes r}\otimes m'_2\otimes 1^{\otimes n-r-2})(\overline{χ^{a_1}ι^{j_1}}\otimes…\otimes\overline{χ^{a_n}ι^{j_n}})=0$} unless $r∈\{x-2,x-1\}$. So
\begin{align*}
Φ_n(\overline{χ^{a_1}ι^{j_1}}&\otimes…\otimes\overline{χ^{a_n}ι^{j_n}})\\
\eqs& (-1)^{n-x+2}f_{n-1}(1^{\otimes x-2}\otimes m'_2\otimes 1^{\otimes n-x} - 1^{\otimes x-1}\otimes m'_2\otimes 1^{n-x-1})\\
  &(\overline{χ^{a_1}ι^{j_1}}\otimes…\otimes\overline{χ^{a_n}ι^{j_n}})\\
\eqs& (-1)^{n-x}f_{n-1}(\overline{χι^{j_1}}\otimes …\otimes\overline{χι^{j_{x-2}}}\otimes m'_2(\overline{χι^{j_{x-1}}}\otimes \overline{ι^{j_x}})\otimes \overline{χι^{j_{x+1}}}\otimes…\otimes \overline{χι^{j_n}}\\
&- \overline{χι^{j_1}}\otimes …\otimes\overline{χι^{j_{x-1}}}\otimes m'_2(\overline{ι^{j_{x}}}\otimes \overline{χι^{j_{x+1}}})\otimes \overline{χι^{j_{x+2}}}\otimes…\otimes \overline{χι^{j_n}})\\
\eqs& (-1)^{n-x}f_{n-1}(\overline{χι^{j_1}}\otimes …\otimes\overline{χι^{j_{x-2}}}\otimes \overline{χι^{j_{x-1}+j_x}}\otimes \overline{χι^{j_{x+1}}}\otimes…\otimes \overline{χι^{j_n}}\\
&- \overline{χι^{j_1}}\otimes …\otimes\overline{χι^{j_{x-1}}}\otimes \overline{χι^{j_{x}+j_{x+1}}}\otimes \overline{χι^{j_{x+2}}}\otimes…\otimes \overline{χι^{j_n}})\\
\eqs& (-1)^{n-x}((-1)^{n-2}γ_{n-1}ι^{j_1+…+j_n} - (-1)^{n-2}γ_{n-1}ι^{j_1+…+j_n}) = 0
\end{align*} 
To show $Ξ_n(\overline{χ^{a_1}ι^{j_1}}\otimes…\otimes\overline{χ^{a_n}ι^{j_n}})=0$, we show $m_2(f_i\otimes f_{n-i})(\overline{χ^{a_1}ι^{j_1}}\otimes…\otimes\overline{χ^{a_n}ι^{j_n}})=0$ for $i∈[1,n-1]$: The element $χ^{a_x}ι^{j_x}$ is a tensor factor of the argument of $f_i$ or of $f_{n-i}$. We write $y=i$ or $y=n-i$ accordingly. Then $y\geq 2$ since $x\notin\{1,n\}$, so $f_y(…) = 0$ and thus $m_2(f_i\otimes f_{n-i})(\overline{χ^{a_1}ι^{j_1}}\otimes…\otimes\overline{χ^{a_n}ι^{j_n}})=0$.\\*
So $Φ_n(\overline{χ^{a_1}ι^{j_1}}\otimes…\otimes\overline{χ^{a_n}ι^{j_n}}) = 0=Ξ_n(\overline{χ^{a_1}ι^{j_1}}\otimes…\otimes\overline{χ^{a_n}ι^{j_n}})$.
\item[Case 2b] Only $a_1=0$, all other $a_x$ equal $1$.\\*
We have $f_{n-1}(1^{\otimes r}\otimes m'_2\otimes 1^{\otimes n-r-2})(\overline{χ^{a_1}ι^{j_1}}\otimes…\otimes\overline{χ^{a_n}ι^{j_n}})=0$ unless $r=0$. So
\begin{align*}
Φ_n(\overline{χ^{a_1}ι^{j_1}}\otimes…\otimes\overline{χ^{a_n}ι^{j_n}})
\eqs&(-1)^nf_{n-1}(1^{\otimes 0}\otimes m'_2\otimes 1^{\otimes n-2})(\overline{χ^{a_1}ι^{j_1}}\otimes…\otimes\overline{χ^{a_n}ι^{j_n}})\\
\eqs&(-1)^n f_{n-1}(m'_2(\overline{ι^{j_1}}\otimes \overline{χι^{j_2}})\otimes \overline{χι^{j_3}}\otimes …\otimes \overline{χι^{j_n}})\\
\eqs&(-1)^n f_{n-1}(\overline{χι^{j_1+j_2}}\otimes \overline{χι^{j_3}}\otimes …\otimes \overline{χι^{j_n}})\\
\eqs& \begin{cases} γ_{n-1}ι^{j_1+…+j_n} & 3\leq n\leq p\\
0 & n\geq p+1 \end{cases}
\end{align*}
We have $(f_i\otimes f_{n-1})(\overline{χ^{a_1}ι^{j_1}}\otimes…\otimes\overline{χ^{a_n}ι^{j_n}})=0$ if $i\geq 2$. So
\begin{align*}
Ξ_n(\overline{χ^{a_1}ι^{j_1}}\otimes…\otimes\overline{χ^{a_n}ι^{j_n}})
\eqs& (-1)^{1\cdot n}m_2(f_1\otimes f_{n-1})(\overline{χ^{a_1}ι^{j_1}}\otimes…\otimes\overline{χ^{a_n}ι^{j_n}})\\
\ovs{\eqref{koszulraw}}& (-1)^nm_2\left( (-1)^{n\cdot |ι^{j_1}|}f_1(\overline{ι^{j_1}})\otimes f_{n-1}(\overline{χι^{j_2}}\otimes…\otimes\overline{χι^{j_n}})\right)\\
\eqs& (-1)^n m_2\left(ι^{j_1}\otimes f_{n-1}(\overline{χι^{j_2}}\otimes…\otimes\overline{χι^{j_n}})\right)\\
\eqs& \begin{cases} (-1)^nm_2(ι^{j_1}\otimes (-1)^{n-2}γ_{n-1}ι^{j_2+…+j_n}) & 3\leq n\leq p\\
0 & n\geq p+1 \end{cases}\\
\eqs& \begin{cases} γ_{n-1}ι^{j_1+…+j_n} & 3\leq n\leq p\\
0 & n\geq p+1 \end{cases}
\end{align*}
So $Φ_n(\overline{χ^{a_1}ι^{j_1}}\otimes…\otimes\overline{χ^{a_n}ι^{j_n}}) =Ξ_n(\overline{χ^{a_1}ι^{j_1}}\otimes…\otimes\overline{χ^{a_n}ι^{j_n}})$.
\item[Case 2c] Only $a_n=0$, all other $a_x$ equal $1$.\\*
Argumentation analogous to case 2b gives
\begin{align*}
Φ_n(\overline{χ^{a_1}ι^{j_1}}\otimes…\otimes\overline{χ^{a_n}ι^{j_n}})
\eqs&(-1)^{2}f_{n-1}(1^{\otimes n-2}\otimes m'_2\otimes 1^{\otimes 0})(\overline{χ^{a_1}ι^{j_1}}\otimes…\otimes\overline{χ^{a_n}ι^{j_n}})\\
\ovs{|m'_2|=0}&f_{n-1}(\overline{χι^{j_1}}\otimes …\otimes \overline{χι^{j_{n-2}}}\otimes m'_2(\overline{χι^{j_{n-1}}}\otimes \overline{ι^{j_n}}))\\
\eqs& \begin{cases} (-1)^{n-2}γ_{n-1}ι^{j_1+…+j_n} & 3\leq n\leq p\\
0 & n\geq p+1 \end{cases}
\end{align*} 
and
\begin{align*}
Ξ_n(\overline{χ^{a_1}ι^{j_1}}\otimes…\otimes\overline{χ^{a_n}ι^{j_n}})
\eqs& (-1)^{n(n-1)}m_2(f_{n-1}\otimes f_{1})(\overline{χ^{a_1}ι^{j_1}}\otimes…\otimes\overline{χ^{a_n}ι^{j_n}})\\
\ovs{|f_1|=0}& m_2\left(f_{n-1}(\overline{χι^{j_1}}\otimes…\otimes\overline{χι^{j_{n-1}}})\otimes f_{1}(\overline{ι^{j_n}})\right)\\
\eqs& \begin{cases} (-1)^{n-2}γ_{n-1}ι^{j_1+…+j_n} & 3\leq n\leq p \\
0& n\geq p+1 \end{cases}
\end{align*}
So $Φ_n(\overline{χ^{a_1}ι^{j_1}}\otimes…\otimes\overline{χ^{a_n}ι^{j_n}}) =Ξ_n(\overline{χ^{a_1}ι^{j_1}}\otimes…\otimes\overline{χ^{a_n}ι^{j_n}})$.
\end{description}
\end{proof}
Now we examine the cases where $a_1=…=a_n=1$:
\begin{lemma}
\label[lemma]{lem:phichi}
For $n\geq 3$, we have $Φ_n(\overline{χι^{j_1}}\otimes…\otimes\overline{χι^{j_n}})=0$ for $\overline{χι^{j_1}}\otimes…\otimes\overline{χι^{j_n}}∈\mathfrak{B}^{\otimes n}=\{\overline{χ^{a_1}ι^{j_1}}\otimes … \otimes \overline{χ^{a_n}ι^{j_n}}  ∈ (\Hm^*A)^{\otimes n} \mid  a_i∈\{0,1\} \text{ and } j_i∈\Z_{\geq 0} \text{ for all }i∈[1,n] \}$. 
\end{lemma}
\begin{proof}
We have $Φ_n(\overline{χι^{j_1}}\otimes…\otimes\overline{χι^{j_n}})=0$ since $Φ_n=\sum_{r=0}^{n-2} (-1)^{n-r} f_{n-1}(1^{\otimes r}\otimes m'_2\otimes 1^{\otimes n-r-2})$ and the argument of $m'_2$ is always of the form $\overline{χι^x}\otimes \overline{χι^y}$, whence its result is zero.
\end{proof}

\begin{lemma}
\label[lemma]{lem:f4}
Condition \eqref{finfrel}$[n]$ holds for $n∈[3,p-1]$ and arguments $\overline{χι^{j_1}}\otimes…\otimes\overline{χι^{j_n}}∈\mathfrak{B}^{\otimes n}=\{\overline{χ^{a_1}ι^{j_1}}\otimes … \otimes \overline{χ^{a_n}ι^{j_n}}  ∈ (\Hm^*A)^{\otimes n} \mid  a_i∈\{0,1\} \text{ and } j_i∈\Z_{\geq 0} \text{ for all }i∈[1,n] \}$.
\end{lemma}
\begin{proof}
For computing $Ξ_n$, we first show that $m_2(f_k\otimes f_{n-k})(\overline{χι^{j_1}}\otimes…\otimes\overline{χι^{j_n}})=0$ for $k∈[2,n-2]$. We will need the following congruence.
\begin{align}
\label{gradesmismatch}
\underbrace{(k(l-1)+l(i+x))}_{\equiv_{p-1} k(l-1)+(p-1)+l(i+x)} - \underbrace{(n-k-1+li')}_{\equiv_{p-1} n-k-1+(p-1)+li'} &\equiv_{p-1} -k+k-n+1 = -(n-1) \nonumber \\&
\not\equiv_{p-1} 0
\end{align}
The last statement results from $2\leq n\leq p-1$.
We set "$\pm$" as a symbol for the (a posteriori irrelevant) signs in the following calculation. For $k∈[2,n-2]$, we have 
\begin{align*}
\hphantom{XXX}m_2&(f_k\otimes f_{n-k})(\overline{χι^{j_1}}\otimes…\otimes\overline{χι^{j_n}})\\
\eqs& \pm m_2((-1)^{k-1} γ_kι^{j_1+…+j_k}\otimes (-1)^{n-k-1}γ_{n-k}ι^{j_{k+1}+…+j_n})\\
\ovs{\substack {j_1+…+j_k=:x,\\j_{k+1}+…+j_n=:y}}&
\pm γ_kι^x\circ γ_{n-k}ι^{y}\\
\eqs& \pm \left(\dum_{i\geq 0} \ls e_k \rs_{k(l-1)+l(i+x)}^{k-1+li}
+ \dum_{i\geq 0} \ls e_{p-k} \rs_{k(l-1) + (p-1) + l(i+x)}^{k-1+(p-1)+li}\right)\\
&\circ \left(\dum_{i'\geq 0} \ls e_{n-k} \rs_{(n-k)(l-1)+l(i'+y)}^{n-k-1+li'}
+ \dum_{i'\geq 0} \ls e_{p-n+k} \rs_{(n-k)(l-1) + (p-1) + l(i'+y)}^{n-k-1+(p-1)+li'}\right)\\
\ovs{\eqref{gradesmismatch}} 0.
\end{align*}
So 
\begin{align*}
Ξ_n(\overline{χι^{j_1}}\otimes…&\otimes\overline{χι^{j_n}})\\
 \eqs& m_2((-1)^nf_1\otimes f_{n-1}+(-1)^{n(n-1)}f_{n-1}\otimes f_1)(\overline{χι^{j_1}}\otimes…\otimes\overline{χι^{j_n}})\\
\eqs& m_2((-1)^{n+n|\overline{χι^{j_1}}|}f_1(\overline{χι^{j_{1}}})\otimes f_{n-1}(\overline{χι^{j_2}}\otimes…\otimes\overline{χι^{j_n}})\\
 &+ f_{n-1}(\overline{χι^{j_1}}\otimes…\otimes\overline{χι^{j_{n-1}}})\otimes f_1(\overline{χι^{j_n}})) \\
\eqs&  m_2(χι^{j_1}\otimes (-1)^{n-2}γ_{n-1}ι^{j_2+…+j_n} + (-1)^{n-2}γ_{n-1}ι^{j_1+…+j_{n-1}}\otimes χι^{j_n})\\
\eqs& (-1)^n (χι^{j_1}\circ γ_{n-1}ι^{j_2+…+j_n} + γ_{n-1}ι^{j_1+…+j_{n-1}}\circ χι^{j_n})\\
\ovs{\text{P.}\ref{pp:iota}(e),\text{L.}\ref{lem:gamma}}& (-1)^n (χ\circ γ_{n-1} + γ_{n-1}\circ χ)\circ ι^{j_1+…+j_n}
\end{align*}
\begin{align*}
χ\circ γ_{n-1}
\eqs& \Big(\dum_{i\geq 0}\left(\ls e_1 \rs_{(i+1)l-1}^{il} + \left\lgroup\dum_{k=1}^{p-2}\ls e_{k+1,k}\rs_{(i+1)l-1+k}^{il + k}\right\rgroup\right.\\
& \left.  + \ls e_{p-1}\rs_{(i+1)l-1 +(p-1)}^{il+(p-1)}
+ \left\lgroup\dum_{k=1}^{p-2}\ls e_{p -k-1,p-k}\rs_{(i+1)l-1+(p-1)+k}^{il +(p-1)+ k}\right\rgroup\right)\Big) \\
&\circ\Big(\dum_{i'\geq 0} \ls e_{n-1} \rs_{(n-1)(l-1)+li'}^{n-2+li'}  
+ \dum_{i'\geq 0} \ls e_{p-n+1} \rs_{(n-1)(l-1) + (p-1) + li'}^{n-2+(p-1)+li'}\Big)\\
\overset{\hphantom{\displaystyle{=}}\mathllap{3\leq n\leq p-1}}{\underset{\hphantom{\displaystyle{=}}\mathllap{\natop{k\leadsto n-1}{i'\leadsto i+1}}}{=}}\eqsp&
\dum_{i\geq 0} \ls e_{n,n-1}\circ e_{n-1} \rs_{(n-1)(l-1)+l(i+1)}^{il+n-1}\\
&+ \dum_{i\geq 0} \ls e_{p-n,p-n+1}\circ e_{p-n+1} \rs_{(n-1)(l-1) + (p-1) + l(i+1)}^{il+p-1+n-1}
\\
\eqs&\dum_{i\geq 0}\left(\ls e_{n,n-1} \rs_{n(l-1) + 1 +li}^{il+n-1}+ \ls e_{p-n,p-n+1}\rs_{n(l-1)+p+li}^{il+p-1+n-1}\right)\\
γ_{n-1}\circ χ \eqs&\Big(\dum_{i'\geq 0} \ls e_{n-1} \rs_{(n-1+i'-1)l+2(p-1)-(n-1)}^{n-2+li'}  
+ \dum_{i'\geq 0} \ls e_{p-n+1} \rs_{(n-1+i')l-n + p}^{n-2+(p-1)+li'}\Big) \\
&\circ\Big(\dum_{i\geq 0}\left(\ls e_1 \rs_{(i+1)l-1}^{il} + \left\lgroup\dum_{k=1}^{p-2}\ls e_{k+1,k}\rs_{(i+1)l-1+k}^{il + k}\right\rgroup\right.\\
& \left.  + \ls e_{p-1}\rs_{(i+1)l-1 +(p-1)}^{il+(p-1)}
 + \left\lgroup\dum_{k=1}^{p-2}\ls e_{p -k-1,p-k}\rs_{(i+1)l-1+(p-1)+k}^{il +(p-1)+ k}\right\rgroup\right)\Big)\\
\ovs{\substack{k\leadsto p-n}}& \dum_{i'\geq 0} \ls e_{n-1} \circ e_{n-1,n} \rs_{(n-1+i')l-1+(p-1)+(p-n)}^{n-2+li'}  \\
&+ \dum_{i'\geq 0} \ls e_{p-n+1}\circ e_{p-n+1,p-n} \rs_{(n+i')l-1+p-n}^{n-2+(p-1)+li'}\\
\eqs&  \dum_{i'\geq 0} \ls e_{n-1,n} \rs_{n(l-1)+ i'l}^{n-2+li'} + \dum_{i'\geq 0} \ls e_{p-n+1,p-n} \rs_{n(l-1) +(p-1)+i'l }^{n-2+(p-1)+li'}
\end{align*}
So $χ\circ γ_{n-1} + γ_{n-1}\circ χ = m_1(γ_n)$ by \cref{lem:gamma}. Therefore
\begin{align*}
Ξ_n(\overline{χι^{j_1}}\otimes…\otimes\overline{χι^{j_n}})
\eqs& (-1)^n  m_1(γ_n)\circ ι^{j_1+…+j_n}
\overset{\text{P.}\ref{pp:iota}(c)}{=} (-1)^n m_1(γ_n ι^{j_1+…+j_n})\\
\eqs& - m_1((-1)^{n-1}γ_n ι^{j_1+…+j_n})\\
\eqs& - m_1\circ f_n(\overline{χι^{j_1}}\otimes…\otimes\overline{χι^{j_n}}).
\end{align*}
We conclude using \cref{lem:finfrelmod} by
\begin{align*}
(f_1\circ m'_n + Φ_n)(\overline{χι^{j_1}}\otimes…\otimes\overline{χι^{j_n}}) \overset{\text{L.}\ref{lem:phichi},\text{D.}\ref{defall}}{=} 0 = (m_1\circ f_n + Ξ_n)(\overline{χι^{j_1}}\otimes…\otimes\overline{χι^{j_n}}).
\end{align*}
\end{proof}
\begin{lemma}
\label[lemma]{lem:f5}
Condition \eqref{finfrel}$[p]$ holds for arguments $\overline{χι^{j_1}}\otimes…\otimes\overline{χι^{j_p}}∈\mathfrak{B}^{\otimes p}=\{\mbox{$\overline{χ^{a_1}ι^{j_1}}\otimes … \otimes \overline{χ^{a_p}ι^{j_p}}$}  ∈ (\Hm^*A)^{\otimes p} \mid  a_i∈\{0,1\} \text{ and } j_i∈\Z_{\geq 0} \text{ for all }i∈[1,p] \}$.
\end{lemma}
\begin{proof}
Recall that $|ι|=l=2(p-1)$ is even, $|χ|=l-1$ is odd and $|f_i| = 1-i$ by \cref{lem:degree}. We have 
\begin{align*}
Ξ_p(\overline{χι^{j_1}}\otimes…&\otimes\overline{χι^{j_p}}) = \dum_{i=1}^{p-1}(-1)^{pi}m_2(f_i\otimes f_{p-i})
(\overline{χι^{j_1}}\otimes…\otimes\overline{χι^{j_p}})\\
\eqs& \dum_{i=1}^{p-1}(-1)^{pi+i(1-(p-i))}m_2(f_i(\overline{χι^{j_1}}\otimes…\otimes\overline{χι^{j_i}})\otimes f_{p-i}(\overline{χι^{j_{i+1}}}\otimes…\otimes\overline{χι^{j_p}}))\\
\eqs& \dum_{i=1}^{p-1}f_i(\overline{χι^{j_1}}\otimes…\otimes\overline{χι^{j_i}})\circ f_{p-i}(\overline{χι^{j_{i+1}}}\otimes…\otimes\overline{χι^{j_p}})\\
\ovs{p\geq 3\vphantom{I_{I_I}}}& χι^{j_1}\circ (-1)^{p-2}γ_{p-1}ι^{j_2+…+j_p} + (-1)^{p-2}γ_{p-1}ι^{j_1+…+j_{p-1}}\circ χι^{j_p}\\
& + \dum_{i=2}^{p-2} (-1)^{i-1}γ_iι^{j_1+…+j_i}\circ (-1)^{p-i-1}γ_{p-i}ι^{j_{i+1}+…+j_p}\\
\ovs{\text{P.}\ref{pp:iota}(b)}& 
(-1)^p\left(χ\circ γ_{p-1} + γ_{p-1}\circ χ + \dum_{k=2}^{p-2} γ_k\circ γ_{p-k}\right)\circ ι^{j_1+…+j_p}
\end{align*}
\begin{align*}
χ\circγ_{p-1}\eqs& \Big(\dum_{i\geq 0}\left(\ls e_1 \rs_{(i+1)l-1}^{il} + \left\lgroup\dum_{k=1}^{p-2}\ls e_{k+1,k}\rs_{(i+1)l-1+k}^{il + k}\right\rgroup\right.\\
&\left. + \ls e_{p-1}\rs_{(i+1)l-1 +(p-1)}^{il+(p-1)}
+ \left\lgroup\dum_{k=1}^{p-2}\ls e_{p -k-1,p-k}\rs_{(i+1)l-1+(p-1)+k}^{il +(p-1)+ k}\right\rgroup\right)\Big)\\
&\circ\left(\dum_{i'\geq 0} \ls e_{p-1} \rs_{(p-1)(l-1)+li'}^{(p-1)-1+li'}+ \dum_{i'\geq 0} \ls e_{1} \rs_{(p-1)(l-1) + (p-1) + li'}^{-1+2(p-1)+li'}\right)\\
\eqs& \dum_{i\geq 0} \ls e_{p-1} \rs_{(p-1)(l-1)+l(i+1)}^{il+(p-1)}+ \dum_{i\geq 0} \ls e_{1} \rs_{(p-1)(l-1) + (p-1) + li}^{il}\\
\eqs& \dum_{i\geq 0} \ls e_{p-1} \rs_{(p+i-1)l+(p-1)}^{il+(p-1)}+ \dum_{i\geq 0} \ls e_{1} \rs_{(p+i-1)l}^{il}
\end{align*} 
\begin{align*}
γ_{p-1}\circ χ\eqs& \left(\dum_{i'\geq 0} \ls e_{p-1} \rs_{(p+i'-2)l+(p-1)}^{(p-1)-1+li'}+ \dum_{i'\geq 0} \ls e_{1} \rs_{(p+i'-1)l}^{-1+2(p-1)+li'}\right) \\
&\circ\Big(\dum_{i\geq 0}\left(\ls e_1 \rs_{(i+1)l-1}^{il} + \left\lgroup\dum_{k=1}^{p-2}\ls e_{k+1,k}\rs_{(i+1)l-1+k}^{il + k}\right\rgroup\right.\\
&\left. + \ls e_{p-1}\rs_{(i+1)l-1 +(p-1)}^{il+(p-1)}
+ \left\lgroup\dum_{k=1}^{p-2}\ls e_{p -k-1,p-k}\rs_{(i+1)l-1+(p-1)+k}^{il +(p-1)+ k}\right\rgroup\right)\Big)\\
\eqs& \dum_{i'\geq 0} \ls e_{p-1} \rs_{(p+i'-1)l-1+(p-1)}^{(p-1)-1+li'}+ \dum_{i'\geq 0} \ls e_{1} \rs_{(p+i')l-1}^{-1+2(p-1)+li'}\\
\eqs& \dum_{i'\geq 0} \ls e_{p-1} \rs_{(p+i'-1)l+p-2}^{p-2+i'l}+ \dum_{i'\geq 0} \ls e_{1} \rs_{(p+i'-1)l+l-1}^{i'l+l-1}
\end{align*}
\begin{align*}
γ_k\circ γ_{p-k}\eqs& \left(\dum_{i\geq 0} \ls e_k \rs_{(i+k-1)l+l-k}^{k-1+li}  
+ \dum_{i\geq 0} \ls e_{p-k} \rs_{(i+k)l+ (p-1)-k}^{k-1+(p-1)+li}\right)\\
&\circ\left(\dum_{i'\geq 0} \ls e_{p-k} \rs_{(p-k)(l-1)+li'}^{p-k-1+li'}  
+ \dum_{i'\geq 0} \ls e_{k} \rs_{(p-k)(l-1) + (p-1) + li'}^{-k+2(p-1)+li'}\right)\\
\eqs&\dum_{i\geq 0} \ls e_k \rs_{(p-k)(l-1)+(p-1)+l(i+k-1)}^{k-1+li}  
+ \dum_{i\geq 0} \ls e_{p-k} \rs_{(p-k)(l-1)+l(i+k)}^{k-1+(p-1)+li}\\
\eqs&\dum_{i\geq 0} \ls e_k \rs_{(p-k+i+k-1)l -(p-k)+(p-1)}^{k-1+li}  
+ \dum_{i\geq 0} \ls e_{p-k} \rs_{(p-k+i+k)l-(p-k)}^{k-1+(p-1)+li}\\
\eqs&\dum_{i\geq 0} \ls e_k \rs_{(p+i-1)l + k-1}^{k-1+li}  
+ \dum_{i\geq 0} \ls e_{p-k} \rs_{(p+i-1)l+ k-1 + (p-1)}^{k-1+(p-1)+li}\,.
\end{align*}
Thus 
\begin{align*}
χ\circ γ_{p-1} +& γ_{p-1}\circ χ + \dum_{k=2}^{p-2} γ_k\circ γ_{p-k}\\
\eqs&\dum_{i\geq 0} \dum_{k=0}^{p-2}\left(\ls e_{k+1} \rs_{(p+i-1)l + k}^{k+li}  
+  \ls e_{p-k-1} \rs_{(p+i-1)l+ k + (p-1)}^{k+(p-1)+li}\right)\\
\eqs& \dum_{i\geq 0} \dum_{k'=0}^{l-1} \ls e_{ω(k')}\rs_{(p-1+i)l+k'}^{k'+li}
\overset{\text{P.}\ref{pp:iota}(a)}{=} ι^{p-1}
\end{align*}
and 
\begin{align*}
Ξ_p(\overline{χι^{j_1}}\otimes…\otimes\overline{χι^{j_p}})=(-1)^pι^{p-1+j_1+…+j_p}\,.
\end{align*}
So we conclude using \cref{lem:finfrelmod} by
\begin{align*}
\begin{array}{rcl}
(f_1\circ m'_p + Φ_p)(\overline{χι^{j_1}}\otimes…\otimes\overline{χι^{j_p}}) &\overset{\text{L.}\ref{lem:phichi},\text{D.}\ref{defall}}{=} &(-1)^p ι^{p-1+j_1+…+j_p} \\
&\overset{\text{D.}\ref{defall}}{=}& (m_1\circ f_p + Ξ_p)(\overline{χι^{j_1}}\otimes…\otimes\overline{χι^{j_p}}).
\end{array}
\end{align*}
\end{proof}
\begin{lemma}
\label[lemma]{lem:f6}
Condition \eqref{finfrel}$[n]$ holds for $n∈[p+1,2(p-1)]$ and arguments\\*
$\mbox{$\overline{χι^{j_1}}\otimes…\otimes\overline{χι^{j_n}}$} ∈\mathfrak{B}^{\otimes n}=\{\overline{χ^{a_1}ι^{j_1}}\otimes … \otimes \overline{χ^{a_n}ι^{j_n}}  ∈ (\Hm^*A)^{\otimes n} \mid  a_i∈\{0,1\} \text{ and } j_i∈\Z_{\geq 0} \text{ for all }i∈[1,n] \}$.
\end{lemma}
\begin{proof}
As $f_k=0$ for $k\geq p$, we have
\begin{align*}
Ξ_n&(\overline{χι^{j_1}}\otimes…\otimes\overline{χι^{j_n}}) = \dum_{k=n-p+1}^{p-1} (-1)^{nk}m_2(f_k\otimes f_{n-k})
(\overline{χι^{j_1}}\otimes…\otimes\overline{χι^{j_n}})
\end{align*}
The right side is a linear combination of terms of the form $γ_k\circ γ_{n-k}$ for $k∈[n-p-1,p-1]$. We have
\begin{align*}
γ_k\circ γ_{n-k} \eqs& \left(\dum_{i\geq 0} \ls e_k \rs_{k(l-1)+li}^{k-1+li}  
+ \dum_{i\geq 0} \ls e_{p-k} \rs_{k(l-1) + (p-1) + li}^{k-1+(p-1)+li}\right)\\
&\circ\left(\dum_{i'\geq 0} \ls e_{n-k} \rs_{(n-k)(l-1)+li'}^{n-k-1+li'}  
+ \dum_{i'\geq 0} \ls e_{p-n+k} \rs_{(n-k)(l-1) + (p-1) + li'}^{n-k-1+(p-1)+li'}\right)
\end{align*}
A necessary condition for that term  to be non-zero is  $k(l-1)\equiv_{p-1}n-k-1$ as $l=2(p-1)$. We have
\begin{align*}
k(l-1)-(n-k-1) \equiv_{p-1}\eqsp& -k-n+k+1 = 1-n
\not\equiv_{p-1} 0,
\end{align*} 
since $p\leq n-1\leq 2(p-1)-1$. So $γ_k\circ γ_{n-k}= 0$ and $Ξ_n(\overline{χι^{j_1}}\otimes…\otimes\overline{χι^{j_n}})=0$.
We conclude using \cref{lem:finfrelmod} by
\begin{align*}
(f_1\circ m'_n + Φ_n)(\overline{χι^{j_1}}\otimes…\otimes\overline{χι^{j_n}}) \overset{\text{L.}\ref{lem:phichi}, \text{D.}\ref{defall}}{=} 0 \overset{\text{D.}\ref{defall}}{=} (m_1\circ f_n + Ξ_n)(\overline{χι^{j_1}}\otimes…\otimes\overline{χι^{j_n}}).
\end{align*}
\end{proof}
One could formulate a lemma similar to \cref{lem:f6} for the case $n> 2(p-1)$ as then the sum $ \sum_{k=n-p+1}^{p-1} (-1)^{nk} m_2(f_k\otimes f_{n-k})
(\overline{χι^{j_1}}\otimes…\otimes\overline{χι^{j_n}})$ is in fact empty. 
Instead we use  \cref{lem:finffinite} to prove \eqref{finfrel}[$n$] for $n>2p-2$: 
\begin{proof}[Proof of \cref{tm:statement}]
\cref{lem:f1,lem:f2,lem:f3,lem:f4,lem:f5,lem:f6} ensure that \eqref{finfrel}[$n$] holds for \mbox{$n∈[1,2p-2]$}. Then \cref{lem:finffinite} with $k=p$ proves that \eqref{finfrel}[$n$] holds for all $n∈[1,∞]$, cf. also \cref{defall}. By \cref{lem:f1inj}, $f_1$ is injective. By \cref{lem:degree}, the degrees are as required in \cref{lem:aaut}. \cref{lem:aaut} proves that $(\Hm^* A, (m'_n)_{n\geq 1})$ is an $\A_∞$-algebra and $(f_n)_{n\geq 1}$ is an $\A_∞$-morphism from  $(\Hm^* A, (m'_n)_{n\geq 1})$ to $(A, (m_n)_{n\geq 1})$. 
By \cref{lem:f1}, we have $m'_1=0$. Thus $(\Hm^* A, (m'_n)_{n\geq 1})$ is a minimal $\A_∞$-algebra. By \cref{lem:f1}, the complex morphism $f_1:(\Hm^* A,m'_1)→(A,m_1)$ is a quasi-isomorphism which induces the identity in homology. So the $\A_∞$-morphism $(f_n)_{n\geq 1}: (\Hm^* A, (m'_n)_{n\geq 1}) → (A, (m_n)_{n\geq 1})$ is a quasi-isomorphism and the proof of \cref{tm:statement} is complete.
\end{proof}



\subsection{At the prime \texorpdfstring{$2$}{2}}
\label{prime2}
We examine the case at the prime $2$.  
We use a direct approach.
Note that $\Sy_2$ is a cyclic group so the theory of cyclic groups applies as well.

We have $\fs\Sy_2= \{0, (\text{id}),(1,2),(\text{id})+(1,2)\}$. We have maps given by
\begin{align*}
\begin{array}{rrcl}
ε:&\fs\Sy_2 & \longrightarrow & \fs\\
& a(\text{id})+b(1,2) &\longmapsto & a+b \\
D:&\fs\Sy_2 & \longrightarrow & \fs\Sy_2\\
& a(\text{id})+b(1,2) &\longmapsto & (a+b)\left((\text{id})+(1,2)\right).
\end{array}
\end{align*}
We see that $ε$ is surjective and $\ker ε = \ker D = \im D = \{0, (\text{id})+(1,2)\}$. The maps $ε$ and $D$ are $\fs\Sy_2$-linear, where $\fs$ is the $\fs\Sy_2$-module that corresponds to the trivial representation of $\Sy_2$. So we have a projective resolution of $\fs$ with augmentation $ε$ by
\begin{align*}
\pres \fs := (\cdots \xrightarrow{D} \underbrace{\fs\Sy_2}_{1} \xrightarrow{D} \underbrace{\fs\Sy_2}_{0}  \rightarrow \underbrace{0}_{-1} \rightarrow \cdots),
\end{align*}
where the degrees are written below. 

We set $e_1$ to be the identity on $\fs \Sy_2$.

Let $A:=\Hom^*_{\fs\Sy_2}(\pres\fs,\pres\fs)$ and let the $\A_∞$-structure on $A$ be  $(m_n)_{n\geq 1}$ (cf. \cref{lem:cai}). Recall the conventions concerning $\Hom_B^k(C,C')$ for complexes $C,C'$ and $k∈\Z$. 
\begin{lemma}
An $\fs$-basis of $\Hm^*A$ is given by $\{\overline{ξ^j} \mid j\geq 0\}$ where 
\begin{align*}
ξ := \dum_{i\geq 0} \ls e_1\rs_{i+1}^i∈A.
\end{align*}
\end{lemma}
\begin{proof}
Straightforward induction yields, for $j\geq 0$,
\begin{align*}
ξ^j = \dum_{i\geq 0} \ls e_1 \rs_{i+j}^i\,.
\end{align*}
We have 
\begin{align*}
m_1(ξ^j) \eqs& d \circ ξ^j - (-1)^j ξ^j \circ d \overset{} = d \circ ξ^j + ξ^j \circ d\\
\eqs& \left(\dum_{i\geq 0} \ls D\rs_{i+1}^i\right)\circ  \left(\dum_{i\geq 0} \ls e_1 \rs_{i+j}^i\right) + \left(\dum_{i\geq 0} \ls e_1 \rs_{i+j}^i\right)\circ \left(\dum_{i\geq 0} \ls D\rs_{i+1}^i\right)\\
\eqs& \dum_{i\geq 0} \ls D\rs_{i+j+1}^i + \dum_{i\geq 0} \ls D\rs_{i+j+1}^i = 0,
\end{align*}
so $ξ^j$ is a cycle.
As $\Hom_{\fs \Sy_2}(\fs \Sy_2, \fs) = \{0,ε\}$ and $ε\circ D=0$, the differentials of $\Hom^*(\pres \fs, \fs)$ (cf.\ \cref{prpr-prm}) are all zero. So $\{ε\}$ is an  $\fs$-basis of $\Hm^k \Hom^*(\pres \fs, \fs)$ for $k\geq 0$. Since in the notion of \cref{prpr-prm}, $\bar{Ψ}_k(\overline{ξ^k}) = ε$, the set $\{\overline{ξ^k}\}$ is an $\fs$-basis of $\Hm^k\Hom^*(\pres\fs,\pres\fs)$ for $k\geq 0$. For $k<0$ we have $\Hm^k\Hom^*(\pres\fs,\pres\fs) \cong \Hm^k \Hom^*(\pres \fs, \fs) = 0$. So $\{\overline{ξ^j} \mid j\geq 0\}$ is an $\fs$-basis of $\Hm^* A$.
\end{proof}
We define families of maps $(f_n:(\Hm^*A)^{\otimes n}→A)_{n\geq 1}$ and $(m'_n:(\Hm^*A)^{\otimes n} → \Hm^* A)_{n\geq 1}$ as follows. $f_1$ and $m'_2$  are given on a basis by 
\begin{align*}
f_1(\overline{ξ^j}) :=\eqsp& ξ^j &&\text{for $j\geq 0$}\\
m'_2(\overline{ξ^j}\otimes \overline{ξ^k}) :=\eqsp& \overline{ξ^{j+k}} &&\text{for $j,k\geq 0$}.
\end{align*}
 All other maps are set to zero.  

It is straightforward to check that $(\Hm^* A, (m'_n)_{n\geq 1})$ is a pre-$\A_∞$-algebra and $(f_n)_{n\geq 1}$ is a pre-$A_∞$-morphism from $\Hm^*A$ to $A$. As $m'_2$ is associative, $(\Hm^* A, (m'_n)_{n\geq 1})$ is a dg-algebra, so in particular an $\A_∞$-algebra.
As $f_k=0$ for $k\neq 1$, \eqref{finfrel}[$n$] simplifies to
\begin{align*}
f_1\circ m'_n \eqs& m_n\circ (\underbrace{f_1\otimes \cdots \otimes f_1}_{n\text{ factors}}).
\end{align*}
As $m'_n=0$ and $m_n=0$ for $n\geq 3$, \eqref{finfrel}[$n$] is satisfied for all $n\geq 3$. For $n∈\{1,2\}$, we have
\begin{align*}
f_1\circ m'_1 \eqs& m_1\circ f_1\\
f_1\circ m'_2 \eqs& m_2(f_1\otimes f_1).
\end{align*}
The second equation follows immediately from the definition of $m'_2$ and $f_1$. The first equation holds as $m'_1=0$ and the images of $f_1$ are all cycles. So \eqref{finfrel}[$n$] holds for all $n$ and  $(f_n)_{n\geq 1}$ is an $\A_∞$-morphism from $(\Hm^* A, (m'_n)_{n\geq 1})$ to $(A, (m_n)_{n\geq 1})$. By the construction of $f_1$,  it induces the identity on homology. Thus $(\Hm^* A, (m'_n)_{n\geq 1})$ is a minimal model of $(A, (m_n)_{n\geq 1})$. 

\begin{bem}[Comparison with primes $p\geq 3$]
\label[bem]{bem:comp}
At a prime $p \geq 3$, we have constructed a projective resolution with period length $l=2(p-1)$ in \eqref{presform1}. 
If one constructs a projective resolution of $\Z_{(2)}$ analogous to the case $p\geq 3$, we have a sequence of the form
\[\cdots \rightarrow \Z_{(2)}\Sy_2 \xrightarrow{\hat e_{2,2}^*} \Z_{(2)}\Sy_2\xrightarrow{\hat e_{2,2}}\Z_{(2)}\Sy_2 \xrightarrow{\hat e_{2,2}^*} \Z_{(2)}\Sy_2\xrightarrow{\hat e_{2,2}}\Z_{(2)}\Sy_2 \rightarrow 0 \rightarrow \cdots\]
with a period length of $2$, where 
\begin{align*}
\hat e_{2,2}\colon&\,(\text{id})\,\longmapsto\, (\text{id})-(1,2)\\
\hat e_{2,2}^*\colon&\,(\text{id})\,\longmapsto\, (\text{id})+(1,2).
\end{align*}
However, modulo $2$ the differentials  $\hat e_{2,2}$ and $\hat e_{2,2}^*$ reduce to the same map $D: \fs \Sy_2\rightarrow \fs \Sy_2$, so we obtain a period length of $1$.

The maps $ι$ resp.\ $χ$ from \cref{pp:iota} may be identified with $ξ^2$ resp.\ $ξ$. This way, the definition of $m'_2$ at the prime $2$ is readily compatible with \cref{defall}. 
\end{bem}

\kommentar{\appendix
\nsection{On the bar construction}

\subsection{Applications. Kadeishvili's algorithm and the minimality theorem.}
In this subsection we will discuss the construction of minimal models of $\A_∞$-algebras. Firstly, \cref{lem:aaut} states that certain pre-$\A_n$-structures and pre-$\A_n$-morphisms that arise in the construction of minimal models are actually $\A_n$-structures and $\A_n$-morphisms. Secondly, we give a proof of \cref{tm:kadeishvili}. We will review Kadeishvili's original proof of \cite{Ka82} as it gives a an algorithm for constructing minimal models which can be used for the direct calculation of examples.
Note that Lefèvre-Hasegawa has given a generalization of the minimality theorem, see \cite[Théorème 1.4.1.1]{Le03}, which we will not cover.
}





\begin{thebibliography}{99}
\bibitem{Be98} {\sc Benson, D.J.\ }{\it Representations and cohomology I: Basic representation theory of finite groups and associative algebras}, Cambridge Univ. Press, 1998 
\bibitem{Bo80} {\sc Bourbaki, N. }{\it Éléments de mathématique - Algèbre - Chapitre 10: Algèbre homologique}, 1980 
\bibitem{Gr74} {\sc Green, J.A. }{\it Walking around the Brauer tree}, Journal of the Australian Mathematical Society, Volume 17, Issue 2, p.\ 197-213, 1974
\bibitem{GuLaSt91} {\sc Gugenheim, V.K.A.M., Lambe, L.A., and Stasheff, J.D., }{\it Perturbation theory in differential homological algebra II}, Illinois J. Math. 35, p.\ 357-373, 1991
\bibitem{Ja78} {\sc James, G.D., }{\it The Representation Theory of the Symmetric Groups}, Lecture Notes in Mathematics 682, Springer-Verlag, 1978
\bibitem{JoLa01} {\sc Johansson, L., and Lambe, L.}{\it Transferring Algebra Structures Up to Homology Equivalence}, Math. Scand. 89, p.\ 181-200, 2001
\bibitem{Ka80} {\sc Kadeishvili, T.V., }{\it On the homology theory of fiber spaces}, Russian Math. Surveys, 35:3, p.\ 231-238, 1980
\bibitem{Ka82} {\sc Kadeishvili, T.V., }{\it Algebraic Structure in the Homologies of an $A(∞)$-Algebra (Russian)}, Bulletin of the Academy of Sciences of the Georgian SSR 108 No 2, p.\ 249-252, 1982
\bibitem{Ka87} {\sc Kadeishvili, T.V., }{\it The functor $D$ for a category of A(∞)-algebras (Russian)}, Bulletin of the Academy of Sciences of the Georgian SSR 125, p.\ 273-276, 1987
\bibitem{Ke01} {\sc Keller, B., }{\it Introduction to A-infinity algebras and modules}, Homology, Homotopy and Applications vol. 3(1), 2001 
\bibitem{Ke01ad} {\sc Keller, B., }{\it Addendum to 'Introduction to A-infinity algebras and modules'}, 2002
\bibitem{Kl10} {\sc Klamt, A., }{\it $\A_∞$-structures on the algebra of extension of Verma modules in the parabolic category $\mathcal{O}$}, Diplomarbeit, 2010
\bibitem{Ku99} {\sc K\"unzer, M., }{\it Ties for the integral group ring of the symmetric group}, Thesis, 1999
\bibitem{Le03} {\sc Lefèvre-Hasegawa, K., }{\it Sur les $\A_∞$-catégories}, Thesis, 2003
\bibitem{Ma02} {\sc Madsen, D.,}{\it Homological aspects in representation theory}, Thesis, 2002
\bibitem{Me99} {\sc Merkulov, S.A., }{\it Strong homotopy algebras of a Kähler manifold}, Int. Math. Res. Notices, Vol. 1999, p.\ 153-164, 1999
\bibitem{Pe71} {\sc Peel, M.H., }{\it Hook representations of the symmetric groups}, Glasgow Mathematical Journal, 1971
\bibitem{Pr84} {\sc Prouté, A., }{\it Algèbres différentielles fortement homotopiquement associatives}, Thèse d'Etat, Université Paris VII, 1984
\bibitem{Ro80} {\sc Roggenkamp, K.\ W., }{\it Integral representations and structure of finite group rings}, Séminaire de mathématiques supérieures 71., Presses de l'Université de Montréal, 1980
\bibitem{St63} {\sc Stasheff, J.D., }{\it Homotopy associativity of $H$-spaces II}, Trans. Amer. Math. Soc. 108, p.\ 293-312
\bibitem{Ve08} {\sc Vejdemo-Johansson, M., }{\it Computation of $\A_∞$-algebras in group cohomology}, Thesis, 2008
\bibitem{Ve082} {\sc Vejdemo-Johansson, M., }{\it A partial $\A_∞$-structure on the cohomology of $C_m\times C_n$}, Journal of Homotopy and Related Structures, vol. 3(1), p.\ 1-11, 2008
\bibitem{Ve10} {\sc Vejdemo-Johansson, M., }{\it Blackbox computation of $\A_∞$-algebras}, Georgian Mathematical Journal Volume 17 Issue 2, p.\ 391-404, 2010
\end{thebibliography}
\end{document}